\documentclass[10pt,reqno]{article}

\usepackage{amsfonts,amssymb,stmaryrd,amsthm}
\usepackage{amsmath,palatino}
\usepackage{graphicx}
\usepackage{hyperref}
\usepackage{yfonts,mathrsfs}
\usepackage{pb-diagram}

\newtheorem{theorem}{Theorem}[section]
\newtheorem{lemma}[theorem]{Lemma}
\newtheorem{proposition}[theorem]{Proposition}

\newtheorem{corollary}[theorem]{Corollary}
\newtheorem{remark}[theorem]{Remark}
 \DeclareMathOperator{\dist}{dist} 
 \DeclareMathOperator{\interior}{int}
 \DeclareMathOperator{\Id}{Id}
 \DeclareMathOperator{\Sp}{Sp}
 \DeclareMathOperator{\clos}{cl}

 \newcommand{\R}{{ \mathbb{R}}}
 \newcommand{\N}{{ \mathbb{N}}}
 \newcommand{\Z}{{ \mathbb{Z}}}

\newcommand{\sC}{{\mathscr{C}}}
\newcommand{\sD}{{\mathscr{D}}}
\newcommand{\sU}{{\mathscr{U}}}
\newcommand{\sJ}{{\mathscr{J}}}
\newcommand{\sG}{{\mathscr{G}}}

\newcommand{\sX}{{\mathscr{X}}}
\newcommand{\sW}{{\mathscr{W}}}
\newcommand{\sV}{{\mathscr{V}}}

\newcommand{\eps}{{\varepsilon}}
\newcommand{\wsigma}{\widetilde{\sigma}}

\numberwithin{equation}{section}

\newcounter{listasscnt}\renewcommand{\thelistasscnt}{A\arabic{listasscnt}}
\newenvironment{listass}{\begin{list}{(\thelistasscnt)}{\usecounter{listasscnt}}}{\end{list}}

\begin{document} 
\begin{sloppypar}

\title{The Poincar\'{e}-Bendixson Theorem and the Non-linear Cauchy-Riemann Equations}

\author{J.B. van den Berg, S. Muna\`{o}, R.C.A.M. Vandervorst\footnote{Department of Mathematics, VU University
Amsterdam, The Netherlands.}}

\maketitle

 \begin{abstract}
In \cite{fiedlermallet}   Fiedler and Mallet-Paret  prove a version of the classical Poincar\'{e}-Bendixson Theorem  for   scalar parabolic equations.     We prove that a similar result holds for bounded solutions of the non-linear Cauchy-Riemann equations.
The latter is an application of an abstract theorem for flows with a(n) (unbounded) discrete Lyapunov function.
\end{abstract}

\section{Introduction}
The classical Poincar\'{e}-Bendixson Theorem describes the asymptotic behavior of flows in the plane. 
The topology of the plane puts severe restrictions on the behavior of limit sets.
The Poincar\'{e}-Bendixson Theorem states for example that if the $\alpha$- and the $\omega$-limit set of a  bounded trajectory of a smooth flow in $\R^2$    does not contain  equilibria, then the limit set is a periodic orbit. Several generalizations of   this theorem have appeared in the literature. For instance the generalization of the  Poincar\'e-Bendixson Theorem  to two-dimensional manifolds, cf.\ \cite{abz}.
In \cite{hajek}  an extension   to continuous (two-dimensional) flows is obtained and 
  \cite{ciesielski} provides a generalization to semi-flows. 
The remarkable result by Fiedler and Mallet-Paret  \cite{fiedlermallet} establishes
an extension of the Poincar\'{e}-Bendixson Theorem to  infinite dimensional dynamical systems
with a positive Lyapunov function. They apply their result
to   \emph{scalar} parabolic equations of the form
\begin{equation}\label{rd}
u_s=u_{xx}+f(x,u,u_x),\quad x\in S^1, f\in C^2.
\end{equation}
In this paper we establish  a version of the Poincar\'{e}-Bendixson Theorem for 
bounded orbits of the nonlinear Cauchy-Riemann equations in the plane.  A bounded orbit of the nonlinear Cauchy-Riemann equations
 is a (smooth) bounded function $u\colon \R\times S^1\to \R^2$, which satisfies 
\begin{equation}
\label{nonlinearCR}
u_s-J\bigl(u_t-F(t,u)\bigr)=0, 
\end{equation}
with $u(s,t)=\bigl(p(s,t),q(s,t)\bigr)$, $s\in \R$, $t\in S^1=\R/\Z.$ Here $F(t,u)$ is a smooth non-autonomous  vector field on $\R^2$ and
 $J$ is the  symplectic matrix 
 $$
 J=\left(
 \begin{array}{cc}
 0 & -1\\
 1 & 0
 \end{array}
 \right).
 $$
 We prove that the asymptotic behavior, as $s$ goes to infinity, of bounded  solutions of Equation (\ref{nonlinearCR}) is as simple
 as the limiting behavior of flows in $\R^2.$ 
 Equation (\ref{nonlinearCR}) arises in many different contexts, in particular in Floer homology literature,
where the vector field has the form $F(t,u)=F_{H}(t,u),$ i.e. $F_H$ is \emph{Hamiltonian}, cf.\  \cite{mcduffsalamon}.
The latter implies that  there exists a time-dependent Hamiltonian function $H(t,\cdot):\R^2\to \R$, such that $F_H(t,u)=J\nabla H(t,u)$.
In the Hamiltonian case the Cauchy-Riemann equations   are the $L^2$-gradient flow of the Hamilton action
and as such
 bounded solutions of (\ref{nonlinearCR}) will, generically,  be   connections orbits between  equilibria.
The Hamilton action is an $\R$-valued Lyapunov function for the Cauchy-Riemann equations.  
In this paper we obtain a result about the asymptotic behavior of orbits for \emph{general} vector fields $F$ in the Cauchy-Riemann equations.

A bounded solution of the Cauchy-Riemann equations is a  smooth function $u$ with $|u(s,t)|\le C$.
Let $X$ be the set of solutions bounded  
 by a fixed (but arbitrary) constant (in the present work we will always choose $C=1$). Endowed with the compact-open topology $X$ is a compact Hausdorff space. The translation invariance of the Cauchy-Riemann equations in the $s-$variable
defines an induced flow 
on $X$ by translating solutions in the $s$-variable.
A bounded solution $u$ can be identified with its orbit $\gamma(u),$ and $\alpha(u)$ and $\omega(u)$ are well-defined elements of $X$.
In Section\ \ref{sectionmainresult} we given a detailed account of the space $X$ and the induced translation flow 
in the context of the Cauchy-Riemann equations.

\begin{theorem}\label{poincare-bendixson1}
Let $u$ be a bounded solution of the   Cauchy-Riemann Equations (\ref{nonlinearCR}).  
Then, for the $\omega$-limit set $\omega(u)$  the following dichotomy holds:
\begin{enumerate}
\item[(i)] either $\omega(u)$ consists of exactly one $s$-periodic orbit, or
\item[(ii)] $\alpha(v)\subseteq E$ and $\omega(v)\subseteq E$, for every $v\in \omega(u)$,
\end{enumerate}
where $E$ denotes the set of equilibria of Equation \eqref{nonlinearCR}, i.e. the 1-periodic solutions of the vector field $F(t,x)$.
The same dichotomy holds for the $\alpha$-limit set $\alpha(u)$.
\end{theorem}
As in the classical Poincar\'{e}-Bendixson Theorem alternative (ii) allows for $\omega(u)$ (or $\alpha(u)$)  to consist of homoclinic and heteroclinic solutions joining equilibria.
An important reason why a generalization of the Poincar\'{e}-Bendixson holds for the Cauchy-Riemann equations is that there exists a continuous projection onto $\R^2$, which is defined as follows.  
Let $t_0\in S^1$ be arbitrary, then define
\begin{equation}\label{pi}
\begin{array}{lccl}
\pi_{t_0}:&X&\to& \R^2\\
        &u=(p,q)&\mapsto& \pi_{t_0}(u)=\bigl(p(0,t_0),q(0,t_0)\bigr).
\end{array}
\end{equation}

\begin{theorem}\label{teorema2}
Under the assumptions of Theorem \ref{poincare-bendixson1} the projection
$$
\pi_{t_0} \colon \omega(u)\to\pi_{t_0}\omega(u)
$$
is a homeomorphism onto its image.
\end{theorem}

In general, if a flow allows a continuous Lyapunov function, then limit sets of orbits consist only of equilibria. Such flows are referred to as
gradient-like flows.
 Theorem \ref{pb} in this paper gives an abstract extension of the Poincar\'e-Bendixson Theorem
 to flows that allow a \emph{discrete} Lyapunov function. 
 In particular Theorem \ref{pb} implies Theorem \ref{poincare-bendixson1}. 
Note that Theorem \ref{teorema2} together with the classical Poincar\'{e}-Bendixson Theorem also implies Theorem \ref{poincare-bendixson1}.
An abstract version of Theorem \ref{teorema2} is proved in Section \ref{sec:strvers}.
 
 The main differences between the results in \cite{fiedlermallet} for parabolic equations and the results in this paper are that the Cauchy-Riemann equations do not define a well-posed initial value problem and, more importantly, the discrete Lyapunov functions that are considered in this paper are \emph{not} bounded from below.
Furthermore, the results obtained in this paper do not assume  differentiability of the flow, nor does the flow need to be defined on a Banach space.
We believe that most of the results in this paper   extendable to semi-flows, e.g.
 \cite{ciesielski}.
 
 In Section \ref{sectionmainresult} we analyze the main properties of the Cauchy-Riemann equations (\ref{nonlinearCR}) and additional details are given in Section \ref{proofoflemmas}. In Section \ref{abstractsetting} we set up an abstract setting which generalizes the properties of the Cauchy-Riemann equations. In Sections \ref{sectionsv} and \ref{sec:strvers} a full proof of the Poincar\'{e}-Bendixson Theorem is given adapted to the abstract setting introduced in Section \ref{abstractsetting}. 

\section{The  Cauchy-Riemann Equations}
\label{sectionmainresult}

The initial value problem of   Equation (\ref{nonlinearCR}) is  ill-posed.
Given an initial value $u(0,t) = u_0(t)$, there may not exist  solutions $u(s,t)$ of Equation (\ref{nonlinearCR})   for any $s$-time interval $I\ni 0$. We therefore
restrict our attention to
 bounded solutions, which are    
 functions $u\in C^1(\R\times S^1;\R^2)$ that satisfy Equation  (\ref{nonlinearCR})  and for which
\begin{equation}
\label{eqn:ent-sol}
  |u(s,t)| <\infty,\quad \text{for all } (s,t) \in \R\times S^1.
\end{equation}
Since each bounded solution may be considered separately, it suffices to look at the space $X$ of functions $u\in C^1(\R\times S^1;\R^2)$ satisfying  Equation  (\ref{nonlinearCR}) and  
$$
|u(s,t)| \le C,\quad \text{for all } (s,t) \in \R\times S^1,
$$
for some fixed arbitrary constant $C>0.$ Note that, without loss of generality, we can choose $C=1.$
On $X$ we consider the compact-open topology, i.e.  
\begin{equation}\label{compactopen}
u^n\xrightarrow{X} u \quad\iff \quad u^n \xrightarrow{C^0_{\text{loc}}} u,
\end{equation}
where the latter indicates uniform convergence on compact subsets of $S^1\times \R.$ 
Since
 $C^0(\R\times S^1;\R^2)$, endowed with the compact-open topology, is   Hausdorff (see \cite[\S 47]{munkres}),  
 and $X \subset C^0(\R\times S^1;\R^2)$, 
 also  $X$ is  a  Hausdorff   space.
\begin{proposition}
 \label{Xcompact}
The solution space $X$ 
is a compact  Hausdorff space.
\end{proposition}
\begin{proof}
 See Section \ref{proofoflemmas}.
 \end{proof}
 
Identify the translation mapping $(s,t) \mapsto (s+\sigma,t)$ by $\sigma \in \R$ and consider the evaluation mapping
\begin{equation}\label{phi}
\R \times C^0(\R\times S^1;\R^2) \to C^0(\R\times S^1;\R^2) ,\quad (\sigma,u) \mapsto \phi^\sigma(u) = u\circ \sigma.
\end{equation}

\begin{lemma}
\label{lem:evalv1}
The evaluation mapping $(\sigma,u) \mapsto \phi^\sigma(u)$
 is continuous with respect to the compact-open topology on $C^0(\R\times S^1;\R^2)$.
 \end{lemma}
 
\begin{proof}
Since $\R\times S^1$ is a locally compact Hausdorff space, the composition of mappings 
$$ C^0(\R\times S^1;\R\times S^1) \times C^0(\R\times S^1;\R^2) \to C^0(\R\times S^1;\R^2),
$$
is continuous with respect to the compact-open topologies on $C^0(\R\times S^1;\R\times S^1)$ and $C^0(\R\times S^1;\R^2)$,
see \cite[\S 46]{munkres}. The translation $\sigma$ as defined above is a continuous mapping in $C^0(\R\times S^1;\R\times S^1)$,
which proves the lemma.
\end{proof}

Since the Cauchy-Riemann Equations are $s$-translation invariant, $u\in X$ implies that
 $\phi^\sigma(u) \in X.$ We therefore obtain a continuous mapping $\R\times X \to X$, again denoted by $\phi^\sigma(u)$.
Also, 
$$
\phi^\sigma\bigl(\phi^{\sigma'}(u)\bigr) = (u\circ \sigma')\circ \sigma = u\circ (\sigma+\sigma') = \phi^{\sigma+\sigma'}(u),
$$
which shows that $\phi^\sigma$ defines a continuous flow on $X$. A continuous flow on $X$ is a continuous mapping $(\sigma,u) \mapsto \phi^\sigma(u) \in X$, such that $\phi^0(u) = u$ and $\phi^{\sigma+\sigma'}(u) =\phi^\sigma(\phi^{\sigma'}(u))$, for all $\sigma,\sigma' \in \R$ and for all $u\in X$. 

Consider the evaluation mapping $\iota: C^0(\R\times S^1;\R^2) \to C^0(S^1;\R^2)$, defined by
$$
u(\cdot,\cdot) \mapsto u(0,\cdot).
$$
By a similar argument as in Lemma \ref{lem:evalv1} it follows that the mapping $\iota$ is a continuous mapping with respect to the compact-open topology on $C^0(S^1;\R^2).$ 
%

\begin{proposition}
\label{prop:aron}
The mapping $\iota: X \to \sX$, with $\sX = \iota(X)$, is a homeomorphism.
\end{proposition}

\begin{proof}
See Section \ref{proofoflemmas}.
 \end{proof}

For $\phi^\sigma$ we have the following commuting diagram:
$$
\begin{diagram}
\node{\R\times X} \arrow{e,l}{\phi^\sigma} \arrow{s,l}{{\rm id} \times \iota}\node{X} \arrow{s,r}{\iota}\\
\node{\R\times \sX} \arrow{e,l}{T^\sigma}   \node{\sX,} 
\end{diagram}
$$
with $u(0,\cdot) \mapsto T^\sigma(u(0,\cdot)) = u(\sigma,\cdot)$, and $T^\sigma$ defines a flow on $\sX$.

The principal tool in the proof of Theorem \ref{poincare-bendixson1} is the existence of an unbounded, discrete Lyapunov function, which decreases   along orbits of the flow $\phi^\sigma.$
Let  $u^1, u^2 \in X$ be two solutions, with $ u^1\not=u^2$,
such that    the function $t\mapsto u^1(s,t)-u^2(s,t)$ is nowhere zero. Then 
define $w:= u^1-u^2\in C^0(\R\times S^1;\R^2).$ 
The  $s$-dependent  winding number $\sW$ of the pair $(u^1,u^2)$ is defined as the winding number of $w$ about the origin, i.e.
\begin{equation}
\label{eqn:wind}
\sW\bigl(u^1(s,\cdot),u^2(s,\cdot)\bigr):=\sW(w(s,\cdot),0)=\frac{1}{2\pi}\int_{S^1} w^*\theta,
\end{equation}
where 
$\theta=\frac{-q dp+p dq}{p^2+q^2}$ is a closed one-form on $\R^2\setminus \{0\}$, cf.\  \cite{GVVW}.
A pair of solutions $(u^1, u^2)\in X\times X$ is said to be \emph{singular} if they belong to the \lq\lq crossing\rq\rq  \ set defined by 
$$
\Sigma_X:=\{(u^1,u^2)\in X\times X : \exists~ s\in \R\ :\ u^1(s,t)=u^2(s,t)\ \text{for some } t\in S^1\}.
$$
The Lyapunov function
 $W\colon (X\times X)\setminus \Sigma_X \to \Z$ is defined by
\begin{equation}
\label{defn:w}
W(u^1,u^2) := \sW\bigl( \iota(u^1),\iota(u^2)\bigr).
\end{equation}
The Lyapunov function $W$ is continuous on $(X\times X)\setminus \Sigma_X$ and constant on connected components.
The set $\Sigma_X$ is a closed  in $X\times X,$ since uniform convergence on compact sets implies point-wise convergence.
The Lyapunov function $W$ is a symmetric:
$$
W(u^1,u^2)=W(u^2,u^1),\ \ \ \  \ \text{for all} \ (u^1,u^2)\not\in \Sigma_{X}.
$$
The \emph{diagonal} in $X\times X$ is defined by
$$
\Delta:=\{(u^1,u^2)\in X\times X: u^1=u^2\},
$$
and $\Delta\subset \Sigma_X$.
The flow $\phi^\sigma$ induces a product flow on $X\times X$ via $(u^1,u^2) \mapsto \bigl(\phi^\sigma(u^1),\phi^\sigma(u^2)\bigr)$,
and the diagonal $\Delta$ is invariant for the product flow.
For the action of the flow on $W$ we have
\begin{eqnarray*}
W\bigl(\phi^\sigma(u^1),\phi^\sigma(u^2)\bigr) &=& \sW\bigl(\iota\circ\phi^\sigma(u^1),\iota\circ\phi^\sigma(u^2)\bigr)\\
&=& \sW\bigl(T^\sigma( \iota(u^1)),T^\sigma(\iota(u^2))\bigr) = \sW(u^1(\sigma,\cdot),u^2(\sigma,\cdot)).
\end{eqnarray*}

In \cite{GVVW} it is proved that the set $\Sigma_{X}\setminus\Delta$ is ``thin" in $X\times X$,
which is the content of the following proposition.

\begin{proposition}[see \cite{GVVW}]
\label{thinness}
For every  singular solution pair  $(u^1,u^2)\in\Sigma_{X}\setminus\Delta$, there exists an $\eps_0 = \eps(u^1,u^2)>0$, such that
$(\phi^{\sigma}(u^1),\phi^{\sigma}(u^2))\not\in\Sigma_{X}$, for all $\sigma\in (-\eps_0,\eps_0)\setminus \{0\}$.
\end{proposition}

Orbits which intersect $\Sigma_{X}$ ``transversely'' (and thus are not in the diagonal) instantly  escape from $\Sigma_X$   and   the diagonal $\Delta$ is the   maximal invariant set contained in $\Sigma_{X}.$
  The following proposition proves that 
  $W$   is a discrete Lyapunov function.  

\begin{proposition}[see \cite{GVVW}]
\label{prop:lyap}
For every pair of singular solutions   $(u^1,u^2)\in\Sigma_{X}\setminus\Delta$, there exists an $\eps_0 = \eps(u^1,u^2)>0$, such that 
$
W(\phi^{\sigma}(u^1),\phi^{\sigma}(u^2))>W(\phi^{\sigma'}(u^1),\phi^{\sigma'}(u^2))$, for all $\sigma\in (-\eps_0,0)$ and all $\sigma' \in (0,\eps_0)$.
\end{proposition}

For a given $u\in X$ define the $\alpha$- and $\omega$-limit sets as:
$$
\begin{array}{rcl}
\omega(u)&:=&\{w\in X: \phi^{\sigma_n}(u)\xrightarrow{X}w, \text{\ for\ some }\ \sigma_n\to\infty\},\\
\alpha(u)&:=&\{w\in X: \phi^{\sigma_n}(u)\xrightarrow{X}w, \text{\ for\ some }\ \sigma_n\to-\infty\}.
\end{array}
$$
The sets $\alpha(u)$ and $\omega(u)$ are closed invariant sets for the flow $\phi^\sigma$, see
\cite[Lemma 4.6 Chapter IV]{hajek}.
 Since $X$ is compact, also $\alpha(u)$ and $\omega(u)$ are compact. Compactness of $X$ also implies  that
 $\alpha(u)$ and $\omega(u)$ are non-empty, see
  \cite[Theorem 4.7 Chapter IV]{hajek}.
  The Hausdorff property of $X$ and the continuity of the flow $\phi^\sigma$ imply
  that $\alpha(u)$ and  $\omega(u)$ are connected sets, see \cite[Theorem 4.7 Chapter IV]{hajek}.
%
Define the equilibria of $\phi^\sigma$ by
$$
E:=\{u\in X:\phi^{\sigma}(u)=u \ \ \text{for all} \ \sigma\in \R\}.
$$
Equilibria are functions $u= u(t)$
  which satisfy the stationary equation
$
u_t=F(t,u).
$

\section{The abstract   Poincar\'{e}-Bendixson Theorem}\label{abstractsetting}

The concepts introduced so far can be 
embedded in a more abstract setting, which generalizes the work by  Fiedler and  Mallet-Paret in \cite{fiedlermallet}.
Let $\phi^\sigma$ be a continuous flow on a \emph{compact} 
 Hausdorff space $X.$ In the case of the Cauchy-Riemann equations the flow $\phi^{\sigma}$ is defined in (\ref{phi}), where the space $X$ is either the full solution space, or  the space which consists of the closure of a single entire (bounded) orbit.
 
 The notions of $\alpha$- and $\omega$-limit sets, defined in Section\ \ref{sectionmainresult} remain unchanged, and 
 $\alpha(u)$ and $\omega(u)$ are non-empty, compact, connected,  invariant sets.

Let $\Delta=\{(u^1,u^2)\in X\times X\colon u^1=u^2\}$ be   invariant  for the product flow induced by $\phi^\sigma.$
We assume that there exist a closed \lq\lq thin\rq\rq\ singular set $\Sigma$, with $\Delta \subset \Sigma\subset X\times X,$
and functions  $W \colon (X\times X)\setminus \Sigma\to \Z$ and  $\pi:X\to\pi(X)\subset\R^2,$ which satisfy the following five axioms:
\begin{listass}
\item\label{continuityofW} the function $W:X\times X\setminus\Sigma\to \Z$, is continuous and symmetric; 
\item\label{Wandpi} the mapping $\pi:X\to \pi(X)\subset\R^2$, is a continuous projection onto its (compact) image;
\item\label{piandsigmageneral} the set 
$\{(u^1,u^2)\in X\times X\colon \pi(u^1)=\pi(u^2)\}$ is a subset of $\Sigma;$ 
\item\label{thingeneral}  for every $(u^1,u^2)\in \Sigma\setminus\Delta,$   there exists an $\eps_0 >0$, depending on $(u^1,u^2)$, such that
$(\phi^{\sigma}(u^1),\phi^{\sigma}(u^2))\not\in\Sigma$,  
for all $\sigma \in (-\eps_0,\eps_0)\setminus \{0\}$;
\item\label{Wdropping} 
for every $(u^1,u^2)\in \Sigma\setminus\Delta,$   there exists an $\eps_0 >0$, depending on $(u^1,u^2)$, such that
$$
W(\phi^{\sigma}(u^1),\phi^{\sigma}(u^2))>W(\phi^{\sigma'}(u^1),\phi^{\sigma'}(u^2)),
$$ 
for all $ \sigma\in(-\eps_0,0)$ and all $\sigma'\in (0,\eps_0)$.
\end{listass}
Axioms (\ref{continuityofW})-
(\ref{Wdropping})
are modeled on the properties of the non-linear Cauchy-Riemann Equations discussed in Section \ref{sectionmainresult}, with $\pi=\pi_{t_0}$ defined in (\ref{pi}).
The above axioms also
  generalize the conditions in the work of  Fiedler and  Mallet-Paret in \cite{fiedlermallet}. Note that the function $W$  is a priori unbounded in the present case and the flow $\phi^\sigma$ does not necessarily regularize. 
Under these assumptions we prove the following Theorem.
\begin{theorem}[Poincar\'{e}-Bendixson]\label{pb}
Let $\phi^\sigma$ be a continuous   flow on a compact   Hausdorff space $X.$ Let $\Sigma$ be a closed subset of $X\times X$, and let  $W\colon (X\times X)\setminus\Sigma\to \Z$  and  $\pi:X\to \pi(X)\subset \R^2$ be mappings as defined above, and which satisfy Axioms (\ref{continuityofW})-(\ref{Wdropping}).
Then, for $\omega(u)$ we have the following dichotomy
\begin{enumerate}
\item[(i)] either $\omega(u)$ consists of precisely one periodic orbit, or else
\item[(ii)] $\alpha(w)\subseteq E$ and $\omega(w)\subseteq E$, for every $w\in\omega(u).$
\end{enumerate}
The same dichotomy holds for $\alpha(u).$
\end{theorem}
As in \cite{fiedlermallet}  the proof of Theorem \ref{pb} will be divided into the three Propositions listed below.
\smallskip

{\em
From this point on
 we assume the hypotheses of Theorem \ref{pb}. 
 }
\begin{proposition}[Soft version]\label{sv}\label{softversion}
Let $u \in X$ and let $w\in \omega(u)$. Then $\omega(w)$ contains a periodic solution or an equilibrium. The same holds for $\alpha(w).$
\end{proposition}
 Proposition \ref{sv} implies that, since $\omega(w)$ and $\alpha(w)$ are both subsets of $\omega(u)$, also $\omega(u)$ contains a periodic solution or an equilibrium.
\begin{proposition}\label{pr2}
Let $u \in X$ and let $w\in\omega(u)$. Then either,
\begin{enumerate}
\item[(i)] $\alpha(w)$ and $\omega(w)$ consist only of equilibria, or else
\item[(ii)] $\gamma(w)$ is a periodic orbit.
\end{enumerate}
\end{proposition}
\begin{proposition}\label{pr3}
Let $u \in X.$ If $\omega(u)$ contains a periodic orbit, then $\omega(u)$ is a single periodic orbit.
\end{proposition}

Proposition \ref{sv} is proved in Section \ref{sectionsv} and the proofs of  Propositions \ref{pr2} and \ref{pr3} are carried out in Section \ref{sectionstv}.  
Proposition \ref{sv} is used in the proof of Proposition \ref{pr2}.
Propositions \ref{pr2} and \ref{pr3} together imply Theorem \ref{pb}. 
Section \ref{sectiontechnicallemmas} contains a number of technical lemmas.

Theorem \ref{pb} is applied directly to the Cauchy-Riemann equations which proves Theorem \ref{poincare-bendixson1}. 
Theorem \ref{teorema2} is proved in Section  \ref{sectionstv}  with a formulation adapted to the abstract setting.
Finally, Section \ref{proofoflemmas} provides the proofs of Propositions \ref{Xcompact} and \ref{prop:aron}.

\section{The soft version}\label{sectionsv}
The hypotheses of  Section \ref{abstractsetting} will be assumed for the remainder of the paper.
%
\begin{lemma}
\label{decreasing lemma}
For every pair $(u^1,u^2) \in (X\times X)\setminus \Delta$
the set $$A_{(u^1,u^2)}:=\{\sigma\in \R\colon \bigl(\phi^\sigma(u^1),\phi^\sigma(u^2)\bigr)\in \Sigma\}$$ consists  of isolated points. Moreover, the mapping
$$
\sigma\mapsto W(\phi^\sigma(u^1),\phi^\sigma(u^2)), 
$$
defined for $\sigma\in\R\setminus A_{(u^1,u^2)},$
is a non-increasing function of $\sigma$ and is constant on the connected components of   $\R\setminus A_{(u^1,u^2)}$. 

\begin{proof}
Suppose  there exists an accumulation point $\sigma_n \to \sigma_*$, with $\sigma_n\in A_{(u^1,u^2)}.$ 
By definition $\bigl(\phi^{\sigma_n}(u^1),\phi^{\sigma_n}(u^2)\bigr) \in \Sigma\setminus \Delta$, since $\Delta$ is invariant and
$(u^1,u^2) \not \in \Delta$. By the continuity of
$\phi^\sigma$ we have 
$$
\bigl(\phi^{\sigma_n }(u^1),\phi^{\sigma_n }(u^2)\bigr)\xrightarrow{n\to \infty} \bigl(\phi^{\sigma_*}(u^1),\phi^{\sigma_*}(u^2)\bigr)\in \Sigma,
$$
since $\Sigma$ is closed. This proves that $\sigma_* \in A_{(u^1,u^2)}$.
The invariance of $\Delta$ implies that $\bigl(\phi^{\sigma_*}(u^1),\phi^{\sigma_*}(u^2)\bigr)\in \Sigma\setminus \Delta$.
By Axiom (A4) there exists an $\eps_0 >0$, depending on $\bigl(\phi^{\sigma_*}(u^1),\phi^{\sigma_*}(u^2)\bigr)$, such that
$(\phi^{\sigma_*+\eps}(u^1),\phi^{\sigma_*+\eps}(u^2))\not\in\Sigma$,  
for all $ \eps \in (-\eps_0,\eps_0)\setminus \{0\}$. This contradicts the fact that $\sigma_*$ is an accumulation point.

The set $A_{(u^1,u^2)}$ is a discrete and ordered set. Let $\sigma'<\sigma''$ be two consecutive points in $A_{(u^1,u^2)}$.
By Axiom (A1), $W$ is continuous and $\Z$-valued, and therefore $W(\phi^\sigma(u^1),\phi^\sigma(u^2))$ is constant on
$\sigma \in (\sigma',\sigma'')$.
By Axiom (\ref{Wdropping}), $W(\phi^\sigma(u^1),\phi^\sigma(u^2))$ drops at 
points in $A_{(u^1,u^2)}$, which shows that $W$ is non-increasing.
\end{proof}
\end{lemma}
%

\begin{lemma}
\label{lem:not-in-sigma}
Let $u\in X$ and $w\in \omega(u).$ For every $w^1,w^2 \in \clos\bigl(\gamma(w)\bigr)$ with $w^1\not=w^2,$  it holds that
$(w^1,w^2) \not\in \Sigma$.
\end{lemma}

\begin{proof} We argue by  contradiction.
Suppose $(w^1,w^2) \in \Sigma\setminus \Delta$, then,
by the Axioms (\ref{thingeneral}) and (\ref{Wdropping}), there exists an $\eps_0>0$ such that 
  $\bigl(\phi^{\sigma}(w^1),\phi^{\sigma}(w^2)\bigr)\not\in\Sigma$
for all $\sigma \in (-\eps_0,\eps_0)\setminus \{0\}$   
  and
  $$
W(\phi^{\sigma}(w^1),\phi^{\sigma}(w^2))>W(\phi^{\sigma'}(w^1),\phi^{\sigma'}(w^2)),
$$ 
for all $\sigma \in (-\eps_0,0)$ and all $\sigma' \in (0,\eps_0)$.
Set $\sigma = -\eps$ and $\sigma' = \eps$, with $0<\eps <\eps_0$.
Since $w^1, w^2\in\clos(\gamma(w)),$ 
there exist $s_1, s_2 \in \R$ such that $\bigl(\phi^{s_1\pm \eps}(w),\phi^{s_2\pm \eps}(w)\bigr) \not \in \Sigma$ and
$\bigl(\phi^{s_1\pm \eps}(w),\phi^{s_2\pm \eps}(w)\bigr)$ is close to 
$\bigl(\phi^{\pm\eps}(w^1),\phi^{\pm\eps}(w^2)\bigr)$.
The continuity of $W$ (Axiom (A1))   implies that 
\begin{equation}\label{dec}
\begin{array}{llllll}
W(\phi^{s_1+\eps}(w),\phi^{s_2+\eps}(w))&=&W(\phi^{\eps}(w^1),\phi^{\eps}(w^2))& & & \\
 &<&W(\phi^{-\eps}(w^1),\phi^{-\eps}(w^2))\\
 &=&W(\phi^{s_1-\eps}(w),\phi^{s_2-\eps}(w)).
\end{array}
\end{equation}
%
Since $\gamma(w)\subset \omega(u)$ is an invariant subset of $\omega(u)$,
the definition of $\omega$-limit set and the continuity of $\phi^{\sigma}$ imply that there exists a sequence $\sigma_n\to \infty$, 
as $n\to \infty$, such that 
\begin{equation}
\label{eqn:conv-3}
\phi^{\sigma_n+ s_1 - s_2  \pm\eps}(u)\to \phi^{s_1\pm\eps}(w)\quad\hbox{and}\quad
 \phi^{\sigma_n\pm\eps}(u)\to \phi^{s_2\pm\eps}(w).
\end{equation}
 Since $\sigma_n$ is divergent  we may assume  
\begin{equation}
\label{eqn:choice-2}
\sigma_{n+1}>\sigma_n+2\eps,\quad \text{ for all } ~n.
\end{equation}
Inequality (\ref{dec}), the convergence in (\ref{eqn:conv-3}), Axiom (A1) (continuity) and the fact that $W$ is locally constant 
(Lemma \ref{decreasing lemma}), imply, for $\sigma_n\to\infty$, that
\begin{eqnarray}
W(\phi^{\sigma_n+s_1-s_2+\eps}(u),\phi^{\sigma_n+\eps}(u))&=&W(\phi^{s_1+\eps}(w),\phi^{s_2+\eps}(w))\nonumber\\
&<&W(\phi^{s_1-\eps}(w),\phi^{s_2-\eps}(w))\nonumber\\
&=&W(\phi^{\sigma_n+s_1-s_2-\eps}(u),\phi^{\sigma_n-\eps}(u)).\nonumber
\end{eqnarray}
Combining the latter with (\ref{eqn:choice-2}) and the fact that $W$ is non-increasing implies that 
$$
W(\phi^{\sigma_{n+1}+s_1-s_2-\eps}(u),\phi^{\sigma_{n+1}-\eps}(u)) <
W(\phi^{\sigma_n+s_1-s_2-\eps}(u),\phi^{\sigma_n-\eps}(u)),
$$
for all $n$.
From this inequality we deduce that $\sigma \mapsto W(\phi^{\sigma+s_1-s_2}(u),\phi^{\sigma}(u))$ has infinitely many jumps and therefore
$$
W(\phi^{\sigma+s_1-s_2}(u),\phi^{\sigma}(u))
\to -\infty,\quad\hbox{as}\quad \sigma \to \infty.
$$
On the other hand, the continuity of $W$  and Equation (\ref{eqn:conv-3}) yield
$$
W(\phi^{\sigma_n+s_1-s_2+\eps}(u),\phi^{\sigma_n+\eps}(u)) = W(\phi^{s_1+\eps}(w),\phi^{s_2+\eps}(w))>-\infty,
$$
as $\sigma_n \to \infty$, 
which is a contradiction.
\end{proof}

\begin{lemma} \label{injectivityofpi1}
Let $u\in X$ and $w\in \omega(u),$ then
$$
\pi\colon \clos\left(\gamma(w)\right)\to \pi\clos(\gamma(w))\subset \R^2
$$ 
is a homeomorphism onto its image. Hence, $\pi\circ\phi^\sigma\circ \pi^{-1}$ is a continuous flow on  $\pi\clos(\gamma(w)).$

\begin{proof}
By Axiom (A2), the projection $\pi\colon \clos\left(\gamma(w)\right)\to \pi\clos(\gamma(w))$ is continuous.
Since $\clos\left(\gamma(w)\right)$ is compact and $ \pi\clos(\gamma(w))$ is Hausdorff, it is sufficient to show that 
$\pi$ is bijective, see  \cite[\S~26, Thm.~26.6]{munkres}.
The projection $\pi $ is   surjective and it remains to show that
 $\pi$ is
injective on $\clos(\gamma(w)).$ 
Suppose   $\pi$ is not injective, 
then there exist
 $w^1, w^2\in\clos(\gamma(w))$, such that $w^1\not=w^2$ and $\pi(w^1)=\pi(w^2).$ 
 Axiom (\ref{piandsigmageneral}) then implies  
that $(w^1,w^2)\in \Sigma\setminus \Delta$. 
On the other hand, Lemma \ref{lem:not-in-sigma} implies that $(w^1,w^2)\not\in \Sigma$,
which is a contradiction. This establishes the injectivity of $\pi$.
\end{proof}
\end{lemma}

For the projected flow on $\pi\clos(\gamma(w))$ we have the following commuting diagram:
\begin{equation}\label{diagramprojflow}
\begin{diagram}
\node{\R\times \clos(\gamma(w))} \arrow{e,l}{\phi^\sigma} \arrow{s,l}{{\rm id} \times \pi}\node{\clos(\gamma(w))} \arrow{s,r}{\pi}\\
\node{\R\times \pi\clos(\gamma(w))} \arrow{e,l}{\psi^\sigma}   \node{\pi\clos(\gamma(w)),} 
\end{diagram}
\end{equation}
where $\psi^\sigma = \pi\circ\phi^\sigma\circ ({\rm id} \times \pi)^{-1}$.
\begin{corollary}
\label{eqtoeqgeneral}
The  equilibria of the planar flow $\psi^\sigma :=\pi\circ\phi^{\sigma}\circ ({\rm id} \times \pi)^{-1}$ 
on $\pi\clos(\gamma(w))$ 
  are in one-to-one correspondence with the equilibria of the   flow $\phi^{\sigma}$ on $\clos(\gamma(w)).$ 
\end{corollary}

Following the natural strategy in proving a Poincar\'{e}-Bendixson type result,
we need to find a transverse curve at a non-equilibrium point and invoke a flow box theorem, ultimately leading to contradiction arguments involving the inside and outside of a Jordan curve made up of a flow line and the transversal. 
Transversals do exist for continuous (but not necessarily smooth) flows in $\R^2$~\cite[section VII.2]{hajek}. However, our flow is defined 
on the closed invariant subset $\pi\clos(\gamma(w))\subset \R^2$.
This set may well have empty interior which prevents us from finding a  section as defined below, and that is also a curve (i.e.\ a so-called transversal).
Roughly speaking, we overcome this difficulty by adapting the usual Jordan curve arguments to a slightly ``less local'' version.

Let $(\sigma,x) \mapsto \psi^\sigma (x)$ be the (local) continuous flow  on the subset $\sD=\pi\clos(\gamma(w))$  of $\R^2$.
A subset $\sC\subset \sD$ is a \emph{section} for $\psi^\sigma$, if there is a $\delta>0$ such that
$$
\psi^{\sigma_1}(\sC)\cap\psi^{\sigma_2}(\sC)=\varnothing, \quad \text{ for all } 0\leq \sigma_1< \sigma_2\leq \delta.
$$
The following lemma shows that for non-equilibrium points $x \in \sD$ there exists a section for the flow in an $\eps$-neighborhood of $x$.  
\begin{lemma}\label{sectionsnoneq}
Let $x\in \sD$ be  a non-equilibrium point of  $\psi^\sigma$.  Then,
\begin{enumerate}
\item[(i)]
for sufficiently small $\delta>0$
there exists a section $\sC$ containing $x$
such that the set 
$$
\sU:=\left\{\psi^\sigma (y) : y \in \sC, \sigma \in\left[-\delta,\delta\right] \right\}
$$
is homeomorphic to $\sC \times [-\delta,\delta]$ via the map $\psi$,
and for $\eps>0$ sufficiently small
\begin{itemize}
\item $B_\eps(x) \cap \sD \subset \sU$,
\item $h_\sigma(\overline{B_\eps(x)} \cap \sU) \in (-\delta,\delta)$.
\end{itemize} 
where $h_\sigma$ is the second components of the inverse homeomorphism $h: \sU \to \sC \times [-\delta,\delta]$, i.e., $h \circ \psi^\sigma(y) = (y,\sigma)$ for all 
$y \in \sC, \sigma \in\left[-\delta,\delta\right]$;
\item[(ii)]
for $\delta_0 < \delta$ sufficiently small  the three balls 
$B^{0} \equiv B_{\eps_0}(x)$, 
$B^{-} \equiv B_{\eps_0}(\psi^{-\delta_0}(x))$ and
$B^{+} \equiv B_{\eps_0}(\psi^{\delta_0}(x))$
are, for $\eps_0 < \eps$ sufficiently small, disjoint subsets of $B_{\eps}(x)$ such that 
$h_\sigma(B^-\cap\sU) < -\frac{\delta_0}{2} < h_\sigma(B^0\cap\sU)  < \frac{\delta_0}{2} <h_\sigma(B^+\cap\sU)$.
Furthermore, for $\eps_1 <\eps_0$ sufficiently small, we have
$\psi^{\pm\delta_0}(y) \in B^{\pm}$ for all $y \in \sC_0 \equiv \sC \cap B_{\eps_1}(x)$. 
\end{enumerate}
\end{lemma}
\begin{proof}
The first part follows from the construction of sections in 
\cite[section VI.2]{hajek}. The second part then follows from continuity of $\psi$ and its inverse $h$.
\end{proof}

The situation described by Lemma~\ref{sectionsnoneq} is illustrated in Figure~\ref{f:canonicdomain}.
\begin{figure}[t]
\centerline{\includegraphics[width=0.7\textwidth]{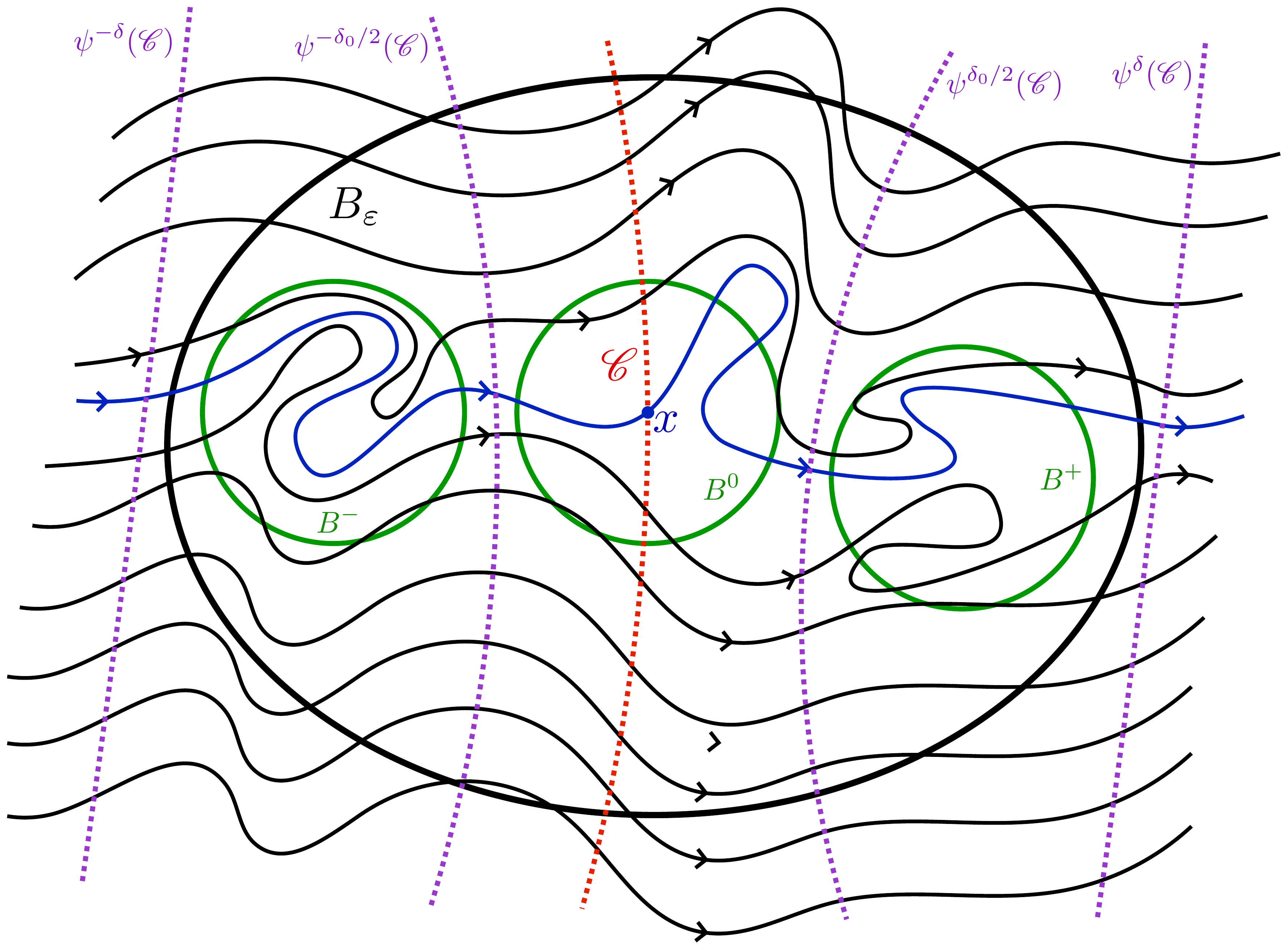}}
\caption{Sketch of the flow in $\sU$ (and the subset $B_\eps(x) \cap \sD$)
through the section $\sC$. The time section $\sigma=\pm\delta_0/2$ separate the
balls $B^0$ and $B^\pm$. We note that the section $\sC$ (and its forward and backward translates in time) are not (necessarily) curves.}
\label{f:canonicdomain}
\end{figure}

\begin{remark} \label{rmk:restrictionFBT}
{\em
In fact, as we will see later, we need to apply a variant of the above lemma
to closed, forward invariant  subsets of the form 
$$
\clos(\gamma(w)\cup\{\phi^{\sigma}(u),\sigma\ge \sigma_*\}),
$$
where $u\in X, w\in\omega(u),\sigma_*\in\R.$ On $\clos(\gamma(w)\cup\{\phi^{\sigma}(u),\sigma\ge \sigma_*\})$ we have a commuting diagram similar to (\ref{diagramprojflow}). In order to have a bi-directional local flow, we define the slightly smaller set
\begin{equation}\label{sV}
\sV:=\pi\clos(\gamma(w)\cup\{\phi^{\sigma}(u),\sigma\ge \sigma_*+2\delta\}),
\end{equation}
for $\delta>0$ small.
Then, if $x\in\sV$ is not an equilibrium for $\psi^{\sigma},$
Lemma \ref{sectionsnoneq} continues to hold with $\sU$ replaced by $\sV$.
}
\end{remark}

\begin{remark}\label{r:J0}
{\em
(i)
The second part of Lemma~\ref{sectionsnoneq} is used to construct a set that replaces
the role of the transversal. Let $y_1$ and $y_2$ be two points in $\sC_0$.
Consider the line segment $L_0$ connecting $y_1$ and
$y_2$. Then $L_0 \subset B^0$. It may happen that
$L_0$ intersects the flow lines of $\psi^{\sigma}(y_{1,2})$ at some $\sigma \neq
0$, but this can be overcome by slightly varying~$\sigma$. Indeed, let the line segment $\ell_0$ be a subset  of
$L_0$ with endpoints $\psi^{\sigma^0_1}(y_1)$ and $\psi^{\sigma^0_2}(y_2)$ for some
$\sigma^0_1,\sigma^0_2\in (-\delta_0/2,\delta_0/2)$, such that $\ell_0$ does not intersect the flow
lines $\psi^{\sigma}(y_{1})$ and $\psi^{\sigma}(y_{2})$ at any other $\sigma \in [-\delta,\delta]$. We still have $\ell_0 \subset B^0$, see Figure~\ref{f:thethreeJs}.

We repeat this
construction in the balls $B^{-}$ and $B^{+}$ to obtain line segments $\ell_-$ and $\ell_+$, respectively, with one end point on each flow line and no other intersections with the flow lines.

Then we obtain three Jordan curves
\begin{alignat*}{1}
  \sJ_0 &= \{ \psi^\sigma (y_1) : \sigma^-_1 \leq \sigma \leq \sigma^+_1 \} \cup 
  \{ \psi^\sigma (y_2) : \sigma^-_2 \leq \sigma \leq \sigma^+_2 \} \cup \ell_- \cup \ell_+ ,\\
  \sJ_{-} &= \{ \psi^\sigma (y_1) : \sigma^-_1 \leq \sigma \leq \sigma^0_1 \} \cup 
  \{ \psi^\sigma (y_2) : \sigma^-_2 \leq \sigma \leq \sigma^0_2 \} \cup \ell_- \cup \ell_0, \\
  \sJ_{+} &= \{ \psi^\sigma (y_1) : \sigma^0_1 \leq \sigma \leq \sigma^+_1 \} \cup 
  \{ \psi^\sigma (y_2) : \sigma^0_2 \leq \sigma \leq \sigma^+_2 \} \cup \ell_0 \cup \ell_+.
\end{alignat*}
in $B_\eps(x)$, see Figure~\ref{f:thethreeJs}.
We denote the interior of $\sJ_j$ by $J_j$, and its exterior by $J_j^*$,
$j \in \{-,0,+\}$. Clearly, $J_\pm \subset J_0$ and 
$J_- \cap J_+ =\varnothing$. 

\begin{figure}
\centerline{\includegraphics[width=0.5\textwidth]{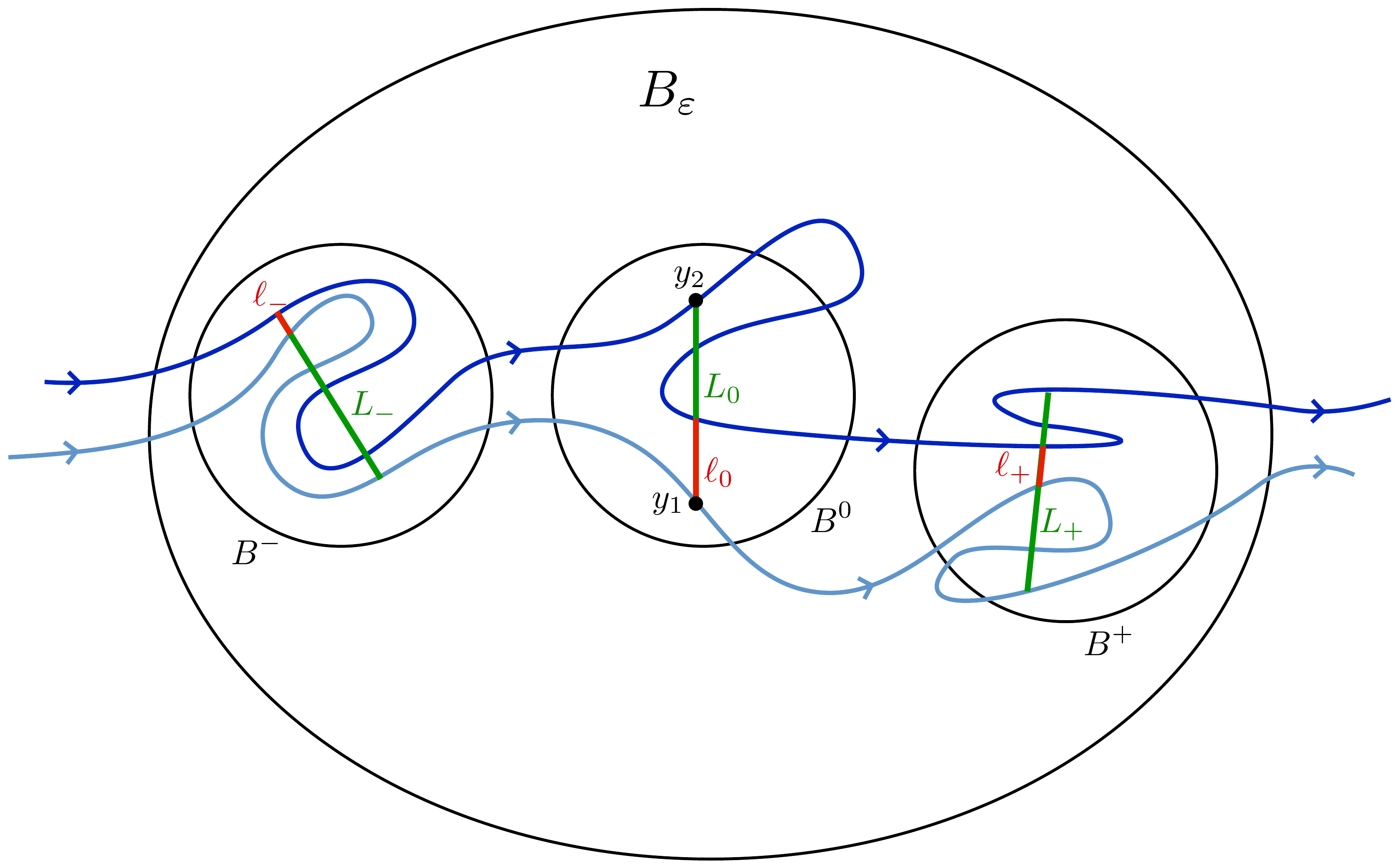}\hspace*{5mm}\raisebox{0.9cm}{\includegraphics[width=0.37\textwidth]{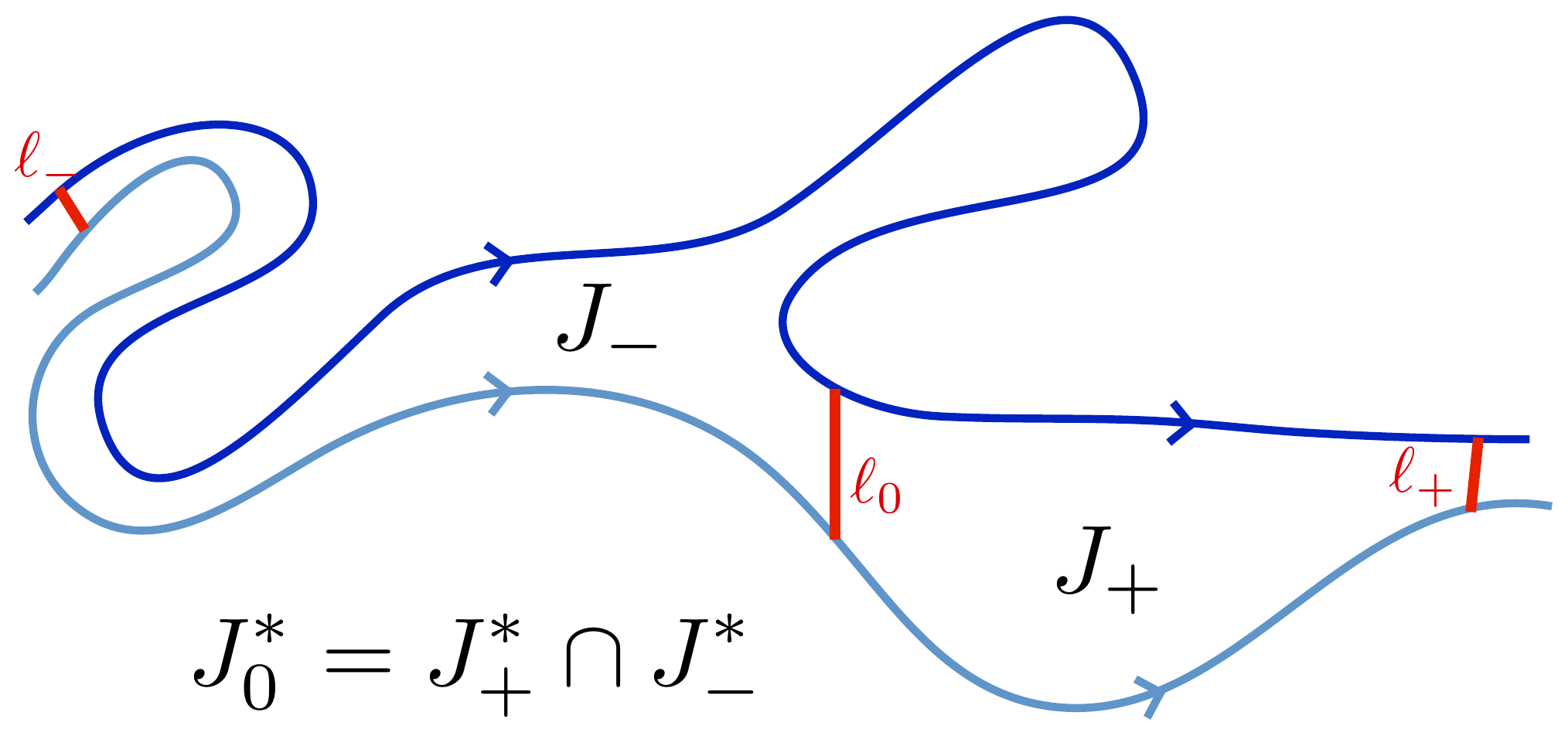}}}
\caption{On the left an illustration of the construction of $\ell_0$ and $\ell_\pm$. On the right a sketch of the interiors $J_0$ and $J_\pm$
of the Jordan curves $\sJ_0$ and $\sJ_\pm$ (as well as their exteriors 
$J_0^*$ and $J_\pm^*$). Note that $J_0 = \interior (\overline{J_- \cup J_+})$.}
\label{f:thethreeJs}
\end{figure}

(ii)
By Lemma~\ref{sectionsnoneq} any flow line in $J_0$ must leave $J_0$ in forward and backward time. By flow invariance of the other boundary components, a flow line can only enter or leave $J_0$ through $\ell_+$ or $\ell_-$.
Moreover, \emph{no} flow line can (in forward time) enter $J_0$ through $\ell_+$ and then leave it through $\ell_-$.
In this sense, the set $J_0$ plays the role of a transversal.
Analogous statements holds for $J_+$ and $J_-$. 
In particular, this implies a slightly stronger statement for the flow in $J_+ \cap J_0$:
if a flow line is in $J_+$ then must leave $J_0$ through $\ell_+$.
Similarly, if a flow line is in $J_-$ then must have entered $J_0$ through $\ell_-$.

(iii)
The flow lines $\{ \psi^{\sigma} (y_{1,2}) : \sigma \in (\sigma^+_{1,2},\delta] \}$
and $\{ \psi^{\sigma} (y_{1,2}) : \sigma \in [-\delta,\sigma^-_{1,2}) \}$
lie in the exterior $J_0^*$. This follows from the fact that by construction they cannot cross $\ell_+$ and $\ell_-$, respectively, and $\psi^{\pm\delta}(y_{1,2})$ all lie in $J_0^*$ by the second bullet of Lemma~\ref{sectionsnoneq}(i). 
}
\end{remark}

\begin{proof}[Proof of Proposition \ref{sv}]
Suppose $\omega(w)$ does not contain any equilibria.
Choose $\zeta\in \omega(w)$ and   $\zeta^*\in \omega (\zeta)$, then 
\begin{equation}
\label{eqn:inclusions}
\omega(\zeta)\subseteq\omega(\omega(w)) = \omega(w) \subseteq \omega(\gamma(w)) = \clos(\gamma(w)).
\end{equation}
Since
  $\zeta^*$ is not an equilibrium,  then
 $\pi (\zeta^*)$ is    not an equilibrium for   $\psi^\sigma = \pi\circ \phi^\sigma \circ ({\rm id}\times \pi)^{-1}$ by  Corollary \ref{eqtoeqgeneral}. 
According to Lemma \ref{sectionsnoneq} there exists a section $\sC$ for  $\psi^\sigma$
through $x= \pi (\zeta^*)$.
Since $\zeta^*\in \omega(\zeta)$  there exist times $\sigma_n \to \infty$ such that $\phi^{\sigma_n} (\zeta) \to \zeta^*$.
By Lemma \ref{sectionsnoneq} these times can be chosen such that $\pi\circ\phi^{\sigma_n} (\zeta)\in \sC_0$, as defined in Lemma~\ref{sectionsnoneq}(ii),
for $\sigma_n$ sufficiently large. Moreover, $\pi\circ\phi^{\sigma}(\zeta)\not\in \sC_0$ for $\sigma\in(\sigma_n,\sigma_{n+1}).$ 
We consider two  cases.

{\it Case 1}. For some $n\not = n'$, we have $\pi\circ\phi^{\sigma_n} (\zeta) = \pi\circ\phi^{\sigma_{n'}} (\zeta)$.
Then, since $\pi$ is a homeomorphism on $\clos\bigl(\gamma(w)\bigr)$
(see Lemma \ref{injectivityofpi1}) and since $\omega(\zeta) \subset \clos\bigl(\gamma(w)\bigr)$ (see Equation (\ref{eqn:inclusions})),  it follows that $\phi^{\sigma_n} (\zeta) = \phi^{\sigma_{n'}} (\zeta)$, and thus $\phi^\sigma(\zeta)$ is a periodic orbit. 

{\it Case 2}. All $\pi\circ\phi^{\sigma_n} (\zeta)$ are mutually distinct.
Take $n$ sufficiently large so that $y_1 \equiv \pi\circ\phi^{\sigma_{n+1}} (\zeta)$ and $y_2 \equiv \pi\circ\phi^{\sigma_{n+1}} (\zeta)$ both lie in $\sC_0$. Denote $\wsigma = \sigma_{n+1}-\sigma_n$, so that $y_2=\psi^{\wsigma}(y_1)$. Apply the construction of Remark~\ref{r:J0}(i) to these $y_1$ and $y_2$.
In addition to $\sJ_0$ we obtain two more Jordan curves
\begin{alignat*}{1}
	\sG_- &= \{  \psi^\sigma (y_1) : \sigma^-_1 \leq \sigma \leq \wsigma + \sigma^-_2 \} \cup \ell_- \\
	\sG_+ &= \{  \psi^\sigma (y_1) : \sigma^+_1 \leq \sigma \leq \wsigma + \sigma^+_2 \} \cup \ell_+ .
\end{alignat*}
Both curves separate $\R^2$ into two open sets, say $A_{\pm}^1$ and $A_{\pm}^2$, see Figure~\ref{f:Apm}. Here, to fix notation, 
we require that $J_0 \subset A_{+}^1$ and $J_0 \subset A_{-}^2$ (recall that $J_0$ is the interior of~$\sJ_0$) so that $A_+^2 \cap A_-^1 =\varnothing$.
It follows from the property of $J_0$ described in Remark~\ref{r:J0}(ii)
and flow invariance of 
$\{  \psi^\sigma (y_1) : \sigma^-_1 \leq \sigma \leq \wsigma + \sigma^+_2 \} $
that once a flow line is in $A_+^2$ it can never enter $A_-^1$ (in forward time).

\begin{figure}[t]
\centerline{\includegraphics[width=6.5cm]{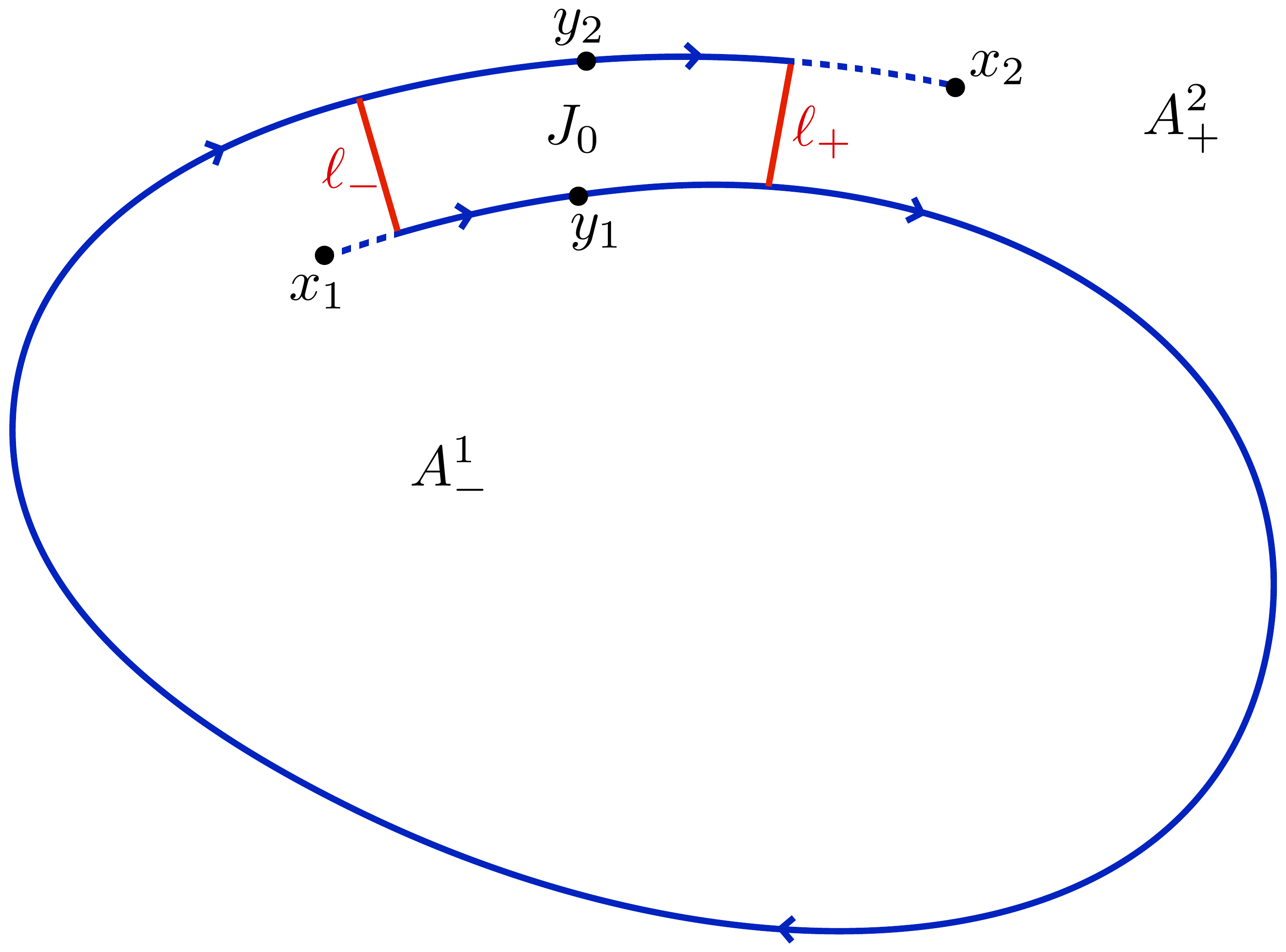}}
\caption{Sketch of the construction of $A_+^2$ and $A_-^1$. Whether $A_-^1$ is a bounded region and $A_+^2$ an unbounded one (as depicted here) or the other way around, is irrelevant for the argument.}
\label{f:Apm}
\end{figure}

Finally, we note that Remark~\ref{r:J0}(iii) implies that 
$x_1=\psi^{-\delta}(y_1) = \pi\circ\phi^{\sigma_{n}-\delta} (\zeta) $ lies in $A_-^1$,
while $x_2=\psi^{\delta}(y_2) = \pi\circ\phi^{\sigma_{n+1}+\delta} (\zeta) $ lies in $A_+^2$.
Now consider the orbit $\pi\circ\phi^\sigma(w)$. 
Since $\zeta\in\omega(w)$ and $\pi$ is continuous, $\pi(\zeta)$ is an $\omega$-limit point of $\pi(w)$ under $\psi^{\sigma}$.
Consequently, the orbit  $\pi\circ\phi^{\sigma}(w)$ 
 keeps (in forward time) visiting arbitrarily small neighborhoods of $x_1 \in A_-^1$ and $x_2 \in A_+^2$. However, as argued above, once a flow line is in $A_+^2$ it can never enter $A_-^1$, which is a contradiction.
\end{proof}

\begin{remark}
\label{rmk:new}
{\em
In \cite[Proposition 2]{fiedlermallet} the ``soft version'' was proved using both smoothness of the flow and fact that there exists a non-negative  discrete Lyapunov function.
The extension given by Proposition \ref{softversion} makes it applicable to the Cauchy-Riemann equations, for which a $\Z$-valued Lyapunov function exists.
}
\end{remark}

\section{The strong version}\label{sec:strvers}
The first subsection contains preliminary lemmas that are used to prove the strong version of the Poincar\'{e}-Bendixson Theorem. The proof of Proposition \ref{pr2}   is carried out in the second subsection.
The arguments in this section resemble those  in \cite{fiedlermallet}, but are adjusted to our setting.

\subsection{Technical lemmas}\label{sectiontechnicallemmas}
\begin{lemma}\label{lemma k} 
Let $u \in X,$
then for every $w\in \omega(u)$ there exists an integer $k(w)$ such that
$$
W(w^1,w^2)=k(w),
$$
for all $w^1, w^2 \in\clos\bigl(\gamma(w)\bigr)$, with $w^1\not=w^2.$
\end{lemma}

\begin{proof}[Proof (cf.\ \cite{fiedlermallet}, Lemma 3.1).]
Since we consider two distinct 
$w^1, w^2\in \clos\bigl(\gamma(w)\bigr)$, we may exclude the case that $w$ is an equilibrium. We 
therefore distinguish two cases: (i) $\gamma(w)$ is a periodic orbit, or (ii) $\sigma \mapsto\phi^\sigma(w)$ is injective.
Lemma \ref{lem:not-in-sigma} implies that $(w^1,w^2) \not \in \Sigma$, and therefore $(w^1,w^2) \mapsto W(w^1,w^2)$ is
a continuous $\Z$-valued function on $\bigl(\clos(\gamma(w)) \times \clos(\gamma(w))\bigr)\setminus \Delta$.

(i) If $\gamma(w)$ is a periodic orbit, then,  $\clos(\gamma(w))=\gamma(w)$, which is homeomorphic to $S^1,$ and $\gamma(w)\times\gamma(w)$ is thus homeomorphic to the 2-torus $\mathbb{T}^2$. 
 Therefore $(w^1,w^2) \mapsto W(w^1,w^2)$ induces 
a continuous $\Z$-valued function on $\mathbb{T}^2 \setminus S^1$. Since the latter is connected, it follows that $W$ is constant
on $\bigl(\gamma(w)\times\gamma(w)\bigr)\setminus \Delta$.

(ii) If $\sigma\to \phi^\sigma(w)$  is injective, then $\bigl(\gamma(w) \times \gamma(w)\bigr)\setminus \Delta$
  has two connected components given by $(\phi^{\sigma_1}(w),\phi^{\sigma_2}(w))$, with $\sigma_1>\sigma_2$, and $\sigma_1<\sigma_2,$ respectively. Since $W$ is symmetric (Axiom (A1)) we conclude that $W$ is constant on $\bigl(\gamma(w) \times \gamma(w)\bigr)\setminus \Delta$.
Note that $\bigl(\clos(\gamma(w)) \times \clos(\gamma(w))\bigr)\setminus \Delta$ is the closure of $\bigl(\gamma(w) \times \gamma(w)\bigr)\setminus \Delta$ in $(X\times X)\setminus \Delta$.
Since $W$ is continuous on
 $\bigl(\clos(\gamma(w)) \times \clos(\gamma(w))\bigr)\setminus \Delta$, it is also constant, which proves the lemma.
\end{proof}

\begin{lemma}\label{lemmak1}
Assume that $u\in X$ and $w\in \omega (u).$ Let $k(w)$ be defined as in Lemma \ref{lemma k}.
 If $\alpha(w)\cap\omega(w)=\varnothing$, then there exists a $\sigma_*\geq 0$, such that
\begin{equation}\label{kappavzero}
W(u^1,w^1)=k(w)
\end{equation}
for every $u^1\in \clos\{\phi^\sigma(u), \sigma\ge \sigma_*\}$ and every $w^1\in \clos (\gamma(w))$, such that $u^1\not=w^1.$ In particular, if $\pi(u^1)=\pi(w^1)$ for some $u^1\in\clos\{\phi^\sigma(u), \sigma\geq \sigma_*\}$ and $w^1\in\clos(\gamma(w))$, then $u^1=w^1.$ Hence 
\begin{equation}\label{firstextinj}
\pi\circ\phi^\sigma (u)\not\in \pi \clos(\gamma(w)) \ \ \text{for all } \sigma\geq \sigma_*.
\end{equation}
\end{lemma}

\begin{proof}[Proof (cf.\ \cite{fiedlermallet}, Lemma 3.2).]
We start by observing that it is enough to prove that (\ref{kappavzero}) holds for $u^1\in\phi^\sigma(u), \sigma\ge \sigma_*.$ Then by continuity of $W,$ the statement follows  for all $u^1\in\clos\{\phi^\sigma(u),\sigma\ge \sigma_*\}.$ 

Suppose there exist sequences $\sigma_n\to \infty,{w}_n\in \clos(\gamma(w)),$ with
$$
\phi^{\sigma_n}(u)\not={w}_n, \quad k_n:=W(\phi^{\sigma_n}(u),{w}_n)\not=k(w).
$$
We may assume, passing to a subsequence if necessary, that for all $n$ we have   either $k_n>k(w)$ or $k_n<k(w).$ We will split the proof in two cases.

{\it Case 1}: $k_n<k(w).$
Again passing to a subsequence if necessary, we may assume that either ${w}_n\in\alpha(w)$ for all $n$ or else ${w}_n\in\clos(\gamma(w))\setminus\alpha(w)$ for all $n.$ Since $\alpha(w)$ and $\omega(w)$ are disjoint by assumption, it follows that $\clos(\gamma(w))\setminus\alpha(w)=\gamma(w)\cup\omega(w).$ 
Choose now $w^1\in \omega(w)$ in case ${w}_n\in \alpha(w),$ and  $w^1\in\alpha(w)$ in case ${w}_n\in\gamma(w)\cup\omega(w).$ In both cases we have $w^1\in\omega(u),$ hence we can choose a sequence $\tilde{\sigma}_n$ with $\tilde{\sigma}_n>\sigma_n,$ for every $n$ such that
$$
w^1:=\lim_{n\to\infty}\phi^{\tilde{\sigma}_n}(u).
$$
In case ${w}_n\in\gamma(w)\cup\omega(w)$ we may assume that $\tilde{\sigma}_n-\sigma_{n}$ is so large that $\phi^{\tilde{\sigma}_n-\sigma_{n}}({w}_n)\in\clos\{\phi^\sigma(w), \sigma>0\}.$ For a further subsequence, we have convergence of $\phi^{\tilde{\sigma}_n-\sigma_n}({w}_n).$ Define
$$
w^2:=\lim_{n\to\infty}\phi^{\tilde{\sigma}_n-\sigma_n}({w}_n).
$$
Note that $w^1,w^2\in\clos(\gamma(w)),$ and $w^1\not=w^2$ since $\alpha(w)\cap\omega(w)=\varnothing.$ In fact, by construction it follows that either $w^1\in\omega(w)$ and $w^2\in\alpha(w),$ or else $w^1\in\alpha(w)$ and $w^2\in\clos\{\gamma(w),\sigma\geq 0\}=\{\phi^\sigma(w), \sigma\geq0\}\cup\omega(w).$
By Lemma \ref{lemma k} there exists $k(w)\in \Z$ such that
$$
W(w^1,w^2)=k(w).
$$
For $n$ large enough the continuity of $W$ implies
\begin{eqnarray}
k_n<k(w)&=&W(w^1,w^2)=W(\phi^{\tilde{\sigma}_n}(u),\phi^{\tilde{\sigma}_n-\sigma_n}({w}_n))\nonumber\\
 &=&W(\phi^{{\sigma}_n+(\tilde{\sigma}_n-\sigma_n)}(u),\phi^{\tilde{\sigma}_n-\sigma_n}({w}^n))\nonumber\\
 &\leq& W(\phi^{{\sigma}_n}(u),{w}_n)=k_n,\nonumber
\end{eqnarray}
which is a contradiction. 

The final assertion  (\ref{firstextinj}) follows from the following observation. Suppose, by contradiction, that there exist a 
$u^1=\phi^{\sigma_1}(u),$ for some $\sigma_1\ge\sigma_*$
 and $w^1\in\clos(\gamma(w)),$ such that $\pi(u^1)=\pi(w^1).$ By what we have just proved, we then have $u^1=w^1.$ 
 Since $w^1\in \clos(\gamma(w))$ and, by assumption, the sets $\alpha(w),$ $\gamma(w)$ and $\omega(w)$ are disjoint, 
 there are only three different possibilities.
\begin{enumerate}
\item $w^1\in\omega(w).$ Then $\phi^{\sigma_1}(u)\in\omega(w).$ By invariance $\omega(u)\subseteq\omega(\omega(w))=\omega(w).$ Since $\alpha(w)\subseteq\omega(u)\subseteq\omega(w),$ this contradicts $\alpha(w)\cap\omega(w)=\varnothing.$
\item $w^1\in\alpha(w).$ Then $\phi^{\sigma_1}(u)\in\alpha(w).$ By invariance $\omega(u)\subseteq\omega(\alpha(w))=\alpha(w).$ Since $\omega(w)\subseteq\omega(u)\subseteq\alpha(w),$ this contradicts $\alpha(w)\cap\omega(w)=\varnothing.$
\item $w^1\in\gamma(w).$ Then $\phi^{\sigma_1}(u)\in\gamma(w).$ By invariance $\omega(u)=\omega(w).$ But $\alpha(w)\subseteq\omega(u)=\omega(w),$ again contradicting $\alpha(w)\cap\omega(w)=\varnothing.$
\end{enumerate}

{\it Case 2}: $k_n>k(w).$ 
This case is analogous to the previous one. It is enough to exchange the roles of $\alpha(w)$ and $\omega(w).$ See \cite[Lemma 3.2]{fiedlermallet}
for further details.
\end{proof}

\begin{remark}\label{r:flowboxextends}
{\em
Lemma \ref{lemmak1} implies that the commutative diagram (\ref{diagramprojflow}) extends from $\clos(\gamma(w))$ to $\clos(\gamma(w)\cup\{\phi^{\sigma}(u),\sigma\ge\sigma_*\}),$ if $\alpha(w)\cap\omega(w)=\varnothing.$ Additionally, by Remark \ref{rmk:restrictionFBT},
the assertions of Lemma \ref{sectionsnoneq}
hold for every $x\in\sV$ (defined in (\ref{sV})) that is not an equilibrium.
}
\end{remark}

\begin{lemma}\label{lemma3.3mp}
Let $u\in X$ and let $\gamma_1$ and $\gamma_2$ be (not necessarily distinct) stationary or periodic orbits in $\omega(u)$. Then, there exists a $k=k(\gamma_1,\gamma_2), k\in\Z$, such that
\begin{equation}\label{3.2}
W(p^1,p^2)=k,
\end{equation}
for every $p^j\in\gamma_j, p^1\not= p^2.$ In particular, the projections of disjoint periodic orbits are disjoint.

 \begin{proof}[Proof (cf.\ \cite{fiedlermallet}, Lemma 3.3).]
We consider the case where $\gamma_1$ and $\gamma_2$ are both periodic, the others are analogous or even simpler. We first claim that $W(p^1,p^2)$ is defined for every $p^1\in\gamma^1$ and every $p^2\in\gamma^2$ with $p^1\not=p^2.$ Suppose, by contradiction, that there exist $p^1\in \gamma^1$ and $p^2\in \gamma^2$ with $p^1\not=p^2$ such that $(p^1,p^2)\in \Sigma\setminus\Delta.$ Then, by Axiom (\ref{thingeneral}) and (\ref{Wdropping}) there exists an $\eps_0>0,$ such that $(\phi^{\sigma}(p^1),\phi^{\sigma}(p^2))\not\in\Sigma$ for every $\sigma\in (-\eps_0,\eps_0)\setminus\{0\}$ and
\begin{equation}\label{ww}
W(\phi^{\sigma'}(p^1),\phi^{\sigma'}(p^2))< W(\phi^{\sigma}(p^1),\phi^{\sigma}(p^2)),
\end{equation}
for $\sigma'\in(0,\eps_0)$ and $\sigma\in(-\eps_0,0).$ Set $\sigma'=\tfrac{\eps_0}{2}$ and $\sigma=-\tfrac{\eps_0}{2}.$
By continuity of $W$ there exists an $\eta\in(0,\tfrac{\eps_0}{2})$ such that $W$ is constant on the set
$$
\mathcal{U}=\left\{(\phi^{\sigma_1}(p^1),\phi^{\sigma_2}(p^2))\ |\  -\tfrac{\eps_0}{2}-\eta<\sigma_1,\sigma_2<\tfrac{\eps_0}{2}+\eta\right\}.
$$
By periodicity of $\gamma^1$ and $\gamma^2$ there is a $\sigma_3>\eps_0$ such that $(\phi^{\sigma_3}(p^1),\phi^{\sigma_3}(p^2))\in\mathcal{U}$ (both in the periodic and the quasi-periodic case). Now, by (\ref{ww})
$$
W(\phi^{\eps_0/2}(p^1),\phi^{\eps_0/2}(p^2))<W(\phi^{-\eps_0/2}(p^1),\phi^{-\eps_0/2}(p^2))=W(\phi^{\sigma_3}(p^1),\phi^{\sigma_3}(p^2)).
$$
Since $\sigma_3>\tfrac{\eps_0}{2},$ this contradicts Lemma \ref{decreasing lemma}. Hence $(p^1,p^2)\not\in\Sigma$ and $W(p^1,p^2)$ is well defined for every $p^1\in\gamma^1$ and every $p^2\in\gamma^2,$ with $p^1\not=p^2.$

This implies, by continuity of $W$, that the map
$$
(p^1,p^2)\to W(p^1,p^2)
$$ 
is locally constant on
$$
\{(p^1,p^2)\in\gamma_1\times\gamma_2\ |\ p^1\not=p^2\}.
$$
This set is connected, which proves (\ref{3.2}). 
\end{proof}
\end{lemma}

\begin{lemma}\label{lemma:new-not-in-sigma}
Let $u\in X$ and $ e \in E.$  For every $w\in\omega(u)$ with $w\not=e$ it holds $(w,e)\not\in \Sigma.$ If, furthermore, $e\not=\omega(u)$ then there exists a $\bar{\sigma}\in\R$ such that the map $\sigma\mapsto W(\phi^{\sigma}(u),e)$ is constant for $\sigma>\bar{\sigma}.$
\begin{proof}
The arguments in this proof resemble those in the proof of Lemma \ref{lem:not-in-sigma}. We repeat the argument. Let $w\in \omega(u).$ Since $w\not=e,$  we can assume that $(w,e)\not\in\Delta.$
Suppose, by contradiction, that $(w,e)\in\Sigma\setminus\Delta,$ then by Axioms (\ref{thingeneral}) and (\ref{Wdropping}), there exists an $\eps_0>0$ such that $(\phi^{\sigma}(w),e)\not\in\Sigma,$ for all $\sigma\in(-\eps_0,\eps_0)\setminus\{0\}$ and
$$
W(\phi^{\sigma}(w),e)> W(\phi^{\sigma'}(w),e),
$$
for all $\sigma\in(-\eps_0,0)$ and all $\sigma'\in (0,\eps_0).$ Set $\sigma=-\eps$ and $\sigma'=\eps,$ with $0<\eps<\eps_0.$ Then we have
\begin{equation}\label{dec2}
W(\phi^{-\eps}(w),e)> W(\phi^{\eps}(w),e).
\end{equation}
 By definition of the $\omega$-limit set and the invariance $\omega$, there exists a sequence $\sigma_n\to\infty,$ as $n\to\infty$ such that
\begin{equation}\label{con}
\phi^{\sigma_n\pm\eps}(u)\to \phi^{\pm\eps}(w).
\end{equation}
Since $\sigma_n$ is divergent we assume that
\begin{equation}\label{eqn:choice3}
\sigma_{n+1}>\sigma_n+2\eps, \quad \text{for all } n\in \N.
\end{equation}
Inequality (\ref{dec2}), convergence in (\ref{con}) and Lemma \ref{decreasing lemma} imply, for $\sigma_n\to\infty,$ that
$$
\begin{array}{ccl}
W(\phi^{\sigma_n+\eps}(u),e)&=& W(\phi^{+\eps}(w),e)\nonumber\\
&<& W(\phi^{-\eps}(w),e)\nonumber\\
&=& W(\phi^{\sigma_n-\eps}(u),e).
\end{array}
$$
Combining the latter with (\ref{eqn:choice3}) and the fact that $W$ is non-increasing, we obtain
$$
W(\phi^{\sigma_{n+1}-\eps}(u),e)<W(\phi^{\sigma_{n}-\eps}(u),e),
$$
for all $n.$ From this, we deduce that $\sigma\mapsto W(\phi^{\sigma}(u),e)$ has infinitely many jumps and therefore
$$
W(\phi^{\sigma}(u),e)\to-\infty \quad \text{as}\quad \sigma\to\infty.
$$
On the other hand, continuity of $W$ and (\ref{con}) imply, for $\sigma_n\to\infty,$ that
$$
W(\phi^{\sigma_n+\eps}(u),e)=W(\phi^{\eps}(w),e)>-\infty,
$$
which is a contradiction.

To prove the final assertion, suppose, by contradiction, that such a $\bar{\sigma}$ does not exist. Then there exists a sequence $\sigma_n\to\infty$ such that $(\phi^{\sigma_n}(u),e)\in\Sigma.$ Now choose a $w\in\omega(u)\setminus\{e\}\not=\varnothing.$ There exists a sequence $\tilde{\sigma}_n\to\infty$ such that $\phi^{\tilde{\sigma}_n}(u)\to w.$ By the first part of the lemma, $W(w,e)\in\Z.$ We may choose $\tilde{\sigma}_n>\sigma_n$ without loss of generality. By continuity of $W$ and axiom (\ref{Wdropping}) it follows that
$$
W(w,e)=\lim_{n\to\infty}W(\phi^{\tilde{\sigma}_n}(u),e)=-\infty,
$$
a contradiction. 
 \end{proof}
\end{lemma}

\begin{lemma}\label{3.4}
Let $u$ be in $X.$  
There exists an integer $k_0\in\Z$ such that
\begin{equation}\label{wweko}
W(w,e)=k_0
\end{equation}
for every $w\in\omega(u),$ and for every equilibrium $e\in \omega(u)$ such that $w\not=e.$ 
\begin{proof}
Fix $e\in E\cap \omega(u).$ Let $w\in\omega(u)\setminus \{e\}.$
According to Lemma \ref{lemma:new-not-in-sigma},  
$W(w,e)$ is well-defined. Since $\phi^{\sigma_n}(u)\to w$ for some $\sigma_n\to\infty,$
\begin{eqnarray}
W(w,e)&=&\lim_{n\to\infty} W(\phi^{\sigma_n}(u),e)\nonumber\\
             &=&\lim_{\sigma\to\infty} W(\phi^{\sigma}(u),e)=k_e,\nonumber
\end{eqnarray}
where the second limit exists by Lemma \ref{lemma:new-not-in-sigma}.
Since the above statement holds for any $w\in\omega(u)\setminus\{e\},$ this implies that $W(w,e)$ is independent of $w\in\omega(u)\setminus\{e\}.$

We still need to show that $W(w,e)$ is independent of $e\in E\cap\omega(u).$
Therefore let $e,\tilde{e}\in E\cap \omega(u), e\not=\tilde{e}.$ Then, by Axiom (\ref{continuityofW}), by the fact that $e,\tilde{e}\in \omega(u),$ and by Lemma \ref{lemma:new-not-in-sigma} it holds that
$$
k_e=W(w,e)  
=W(\tilde{e},{e}) 
        =  W(e,\tilde{e}) = 
         W(w,\tilde{e})=k_{\tilde{e}}.
$$
This shows (\ref{wweko}) and concludes the proof.
\end{proof}
\end{lemma}

\subsection{Proof of the strong version}\label{sectionstv}
In this subsection we prove Propositions \ref{pr2} and \ref{pr3}, which completes the proof of Theorem \ref{pb}.  Theorem \ref{teorema2} 
follows as a consequence of Proposition \ref{th2fmp}.

\begin{proof}[Proof of Proposition \ref{pr2}]
	Let $u \in X$ and $w \in \omega(u)$.	
Suppose, by contradiction, that there is a non-equilibrium $w^*\in\omega(w)$ and that $\gamma(w)$ is not periodic. Lemma \ref{injectivityofpi1} implies  that $\pi\circ\phi^\sigma$ is a planar flow on the set $\omega(w)\subseteq\clos(\gamma(w)).$ By Corollary \ref{eqtoeqgeneral} the point $\pi (w^*)$ is not an equilibrium for $\pi\circ\phi^\sigma.$
According to Lemma \ref{sectionsnoneq} there exist a section $\sC$ through $\pi(w^*)$.
Consider first $\pi\circ\phi^\sigma(w)$ and recall that by Lemma \ref{injectivityofpi1} the map $\sigma\to\pi\circ\phi^\sigma(w) $ is one-to-one since   $\gamma(w)$ is not periodic. Let $\sigma_n\to\infty$ denote 
those positive times for which $\pi\circ\phi^{\sigma_n}(w)\in \sC$. Note that $\{\pi\circ\phi^{\sigma_n}(w)\}_{n=1}^{\infty}$ are all distinct and for
 $n$ sufficiently large $y_1=\pi\circ\phi^{\sigma_n}(w)$ and $y_2 = \pi\circ\phi^{\sigma_{n+1}}(w)$ both lie in $\sC_0$. 
Denote $\wsigma = \sigma_{n+1}-\sigma_n$, so that $y_2=\psi^{\wsigma}(y_1)$.
We apply the construction of Remark~\ref{r:J0}(i) to these $y_1$ and $y_2$.
In addition to $\sJ_0$ and $\sJ_\pm$ we obtain three more Jordan curves (the first two are the same as in the proof of Proposition~\ref{softversion})
\begin{alignat*}{1}
	\sG_- &= \{  \psi^\sigma (y_1) : \sigma^-_1 \leq \sigma \leq \wsigma + \sigma^-_2 \} \cup \ell_- \\
	\sG_+ &= \{  \psi^\sigma (y_1) : \sigma^+_1 \leq \sigma \leq \wsigma + \sigma^+_2 \} \cup \ell_+ \\
	\sG_0 &= \{  \psi^\sigma (y_1) : \sigma^0_1 \leq \sigma \leq \wsigma + \sigma^0_2 \} \cup \ell_0 .
\end{alignat*}
These three curves separate $\R^2$ into two open sets, say $A_{j}^1$ and $A_{j}^2$, with $j\in \{-,0,+\}$. To fix notation,
we require that $J_0 \subset A_{+}^1$ and $J_0 \subset A_{-}^2$ and $J_+ \subset A_+^2$, see Figure~\ref{f:strongversion}.
In particular, this implies that $A_+^2 \subset A_0^2 \subset A_-^2$  and  $A_+^2 \cap A_-^1 =\varnothing$,
as well as  $A_+^2 \cap A_0^1 = \varnothing $ and $A_-^1 \cap A_0^2 =  \varnothing$.
It follows from the properties of $J_0$ and $J_\pm$ described in Remark~\ref{r:J0}(ii)
and invariance of 
$\{  \psi^\sigma (y_1) : \sigma^-_1 \leq \sigma \leq \wsigma + \sigma^+_2 \} $
that in \emph{forward} time once a flow line is in $A_+^2$ it can never enter $A_0^1$,
while  in \emph{backward} time once a flow line is in $A_-^1$ it can never enter $A_0^2$.
We note that Remark~\ref{r:J0}(iii) implies that 
$x_1=\pi\circ\phi^{\sigma_{n}-\delta} (w) $ lies in $A_-^1$,
while $x_2=\pi\circ\phi^{\sigma_{n+1}+\delta} (w) $ lies in $A_+^2$.
Therefore,
$\pi\omega(w) \subset A_0^2$, while $\pi\alpha(w) \subset A_0^1$, cf.\   Figure~\ref{f:strongversion}.
Hence $ \pi\omega(w) \cap \pi\alpha(w) = \varnothing$.
We infer from Lemma \ref{injectivityofpi1} that $\alpha(w)\cap\omega(w)=\varnothing$.

\begin{figure}[t]
\centerline{\includegraphics[width=9cm]{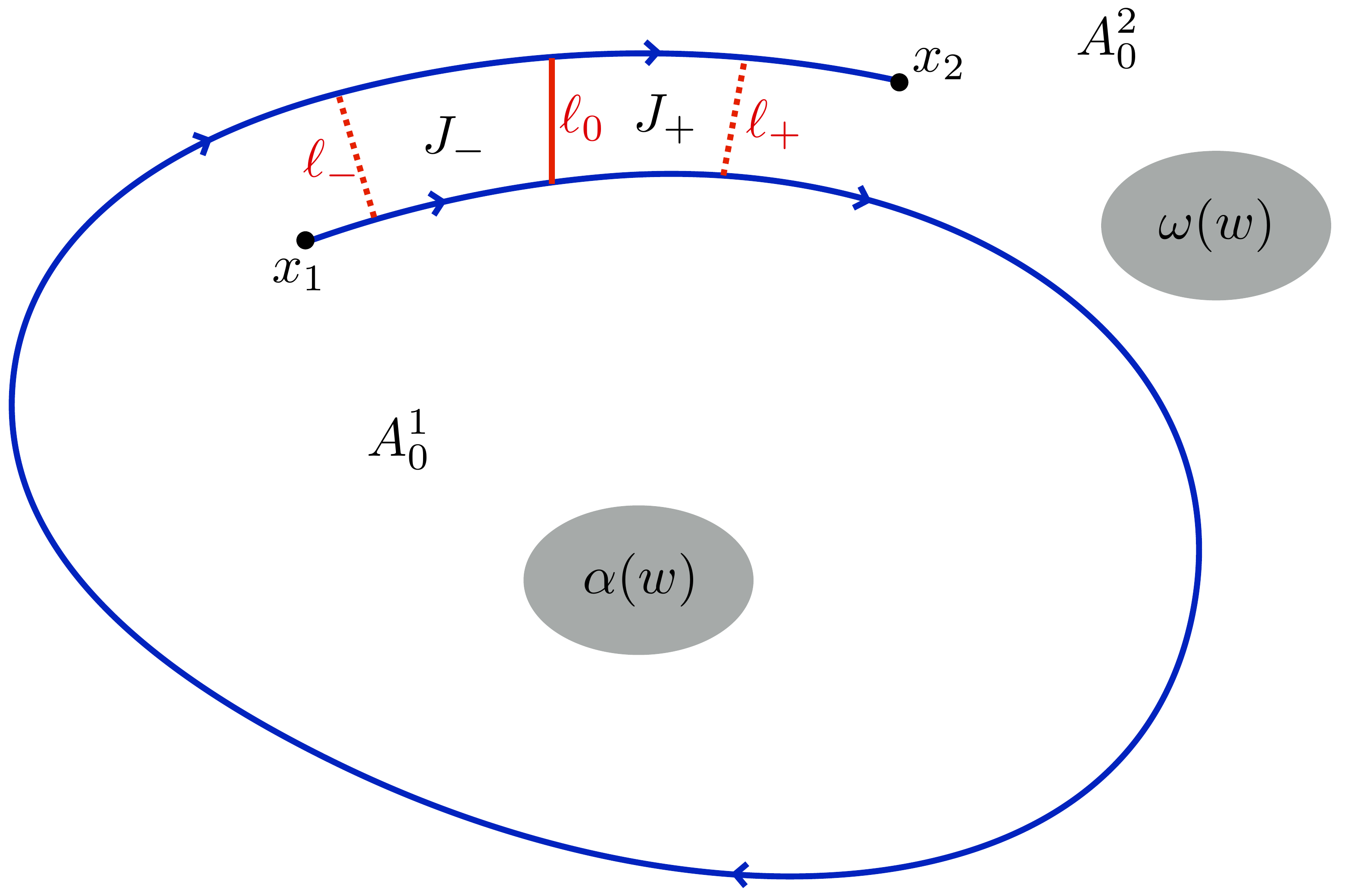}}
\caption{Sketch of the construction of $A_\pm^0$. Note that $J_+=A_-^2 \cap A_0^2$ and $J_-=A_+^1 \cap A_0^1$.}
\label{f:strongversion}
\end{figure}

Next we consider the orbit of $u$.
The assumptions of Lemma~\ref{lemmak1} are  satisfied and hence there exists a time $\sigma_*$, such that the curve $\{ \pi\circ \phi^\sigma(u) : \sigma\ge \sigma_* \}$ cannot cross 
the curve $\pi\circ\phi^\sigma (w)$. Furthermore, it follows from Remarks~\ref{rmk:restrictionFBT} and~\ref{r:flowboxextends} and the above construction, that 
once the flow line $\pi\circ\phi^{\sigma}(u)$ is in $A_+^2$ it can never enter $A_0^1$ (in forward time).
Moreover, by Remark~\ref{r:J0}(ii), once a flow line is in $A_0^2$ then it must enter $A_+^2$ in forward time, after which it can no longer enter $A_0^1$.  
On the other hand, since both $\omega(w)$ and $\alpha(w)$ are contained in $\omega(u)$, the forward orbit 
$\pi\circ \phi^\sigma(u)$ will have $\omega$-limit points when $\sigma\to
\infty$ in both $\pi\alpha(w) \subset A_0^1$ and $\pi\omega(w) \subset A_0^2$. 
This is a contradiction.
\end{proof}

\begin{proof}[Proof of Proposition \ref{pr3} (cf.\ 
 \cite{fiedlermallet}, Proposition 2).]
Suppose that $\omega(u)$ contains  a periodic orbit $\gamma(p)$ as a strict subset. Let $V\subseteq X$ be a closed tubular neighborhood of $\gamma(p).$ Choose $V$ small enough such that it does not contain equilibria and such that $\omega(u)$ still has elements outside $V.$ Since there are accumulation points (for $\phi^\sigma(u)$ when $\sigma$ goes to infinity) both inside and outside $V,$ then $\phi^\sigma(u)$ must enter and leave $V$ infinitely often. Let $\sigma_n\to\infty$ be a sequence such that 
$$
p=\lim_{n\to\infty}\phi^{\sigma_n}(u)
$$ 
and such that $\phi^\sigma(u)$ leaves $V$ between any two consecutive times $\sigma_n.$ Let $I_n:=[\sigma_n-\alpha_n,\sigma_n+\beta_n]$ be the maximal time interval containing $\sigma_n$ such that
$$
\phi^\sigma(u)\in V \ \ \ \text{for all }\sigma\in I_n.
$$
Since $\partial V$ is closed, we may assume convergence (passing to a subsequence, if necessary) of $\phi^{\sigma_n-\alpha_n}(u).$ Note that $\sigma_{n-1}<\sigma_{n}-\alpha_{n}$ thus $\sigma_n-\alpha_n\to\infty.$
Let
$$
q:=\lim_{n\to\infty}\phi^{\sigma_n-\alpha_n}(u)\in\omega(u),
$$
and  $q\in\partial V.$ Moreover we may assume that $\alpha_n+\beta_n\to\infty$ (at least for a subsequence) since $\omega(u)$ contains a periodic orbit in the interior of $V.$ We have thus
$$
\omega(q)\subseteq \clos(\phi^\sigma(q))\subseteq V,\ \  \sigma>0.
$$
From Proposition \ref{pr2} we conclude that $\gamma(q)$ is periodic. By construction $\gamma(q)$ and $\gamma(p)$ are distinct and $\gamma(q)$ is contained in $V.$  By continuity of the flow and the projection $\pi$ and by the compactness of $V$, the sets $\pi\gamma(p)$ and $\pi\gamma(q)$ are close in the Hausdorff metric (of compact subsets of $\R^2$)
provided that we take the tubular neighborhood $V$ sufficiently small. From this it follows that
$\pi \gamma(q)$ and $\pi\gamma(p)$ are nested closed curves.
Reducing $V$ to separate $\gamma(p)$ from $\gamma(q),$ a periodic solution $\gamma(r)$ can be constructed in the same way. Note once more that $\pi \gamma(q), \pi\gamma(p)$ and $\pi\gamma(r)$ are nested closed curves.
Applying Lemma \ref{lemma3.3mp} to the trajectories $\gamma(p)$ and $\gamma(q)$ we conclude
that there exists a $k\in \Z$ such that
$$
W(p^1,q^1)=k,
$$
for all $p^1\in\gamma(p)$ and $q^1\in\gamma(q).$ By continuity of $W$ (Axiom (\ref{continuityofW})) this implies that
$$
W(p^1,\phi^{\sigma_n-\alpha_n}(u))=k,
$$ 
for all $p^1\in\gamma(p)$ when $n$ is big enough, since $\phi^{\sigma_n-\alpha_n}(u)\to q\in\gamma(q).$
By Assumption (\ref{Wdropping}) we get $\pi\circ\phi^\sigma(u)\not\in\pi\gamma(p)$ for every $\sigma$ in the open interval with endpoints $\sigma_n-\alpha_n, \sigma_m-\alpha_m,$ provided $n,m$ are chosen large enough. Since $\sigma_m-\alpha_m\to\infty,$ as $m\to\infty,$ it follows that $\pi\circ\phi^{\sigma}(u)\not\in\pi\gamma(p),$ for any $\sigma$ large enough.
In an analogous manner we can prove that, for $\sigma$ large enough, the curve $\pi\circ\phi^\sigma(u) $ can never intersect $\pi\gamma(q)$ and $\pi{\gamma}(r),$ but this is a contradiction since $\pi\circ\phi^\sigma(u)$ has $\omega$-limit points as $\sigma\to\infty$ in the three nested curves $\pi{\gamma}(p), \pi{\gamma}(q),\pi{\gamma}(r).$
\end{proof} 

\begin{proposition}\label{th2fmp}
Let $u\in X$. Then,
$$
\pi :\omega(u)\to\pi(\omega(u))
$$
is a homeomorphism onto its image. Hence $\pi\circ\phi^\sigma$ is a flow on $\pi(\omega(u)).$

\begin{proof}[Proof (cf.\ \cite{fiedlermallet}, Theorem 2).]
By Axiom (\ref{Wdropping}) it is enough to show that there exists a $k_0\in \Z$ such that
\begin{equation}\label{6.1}
W(w^1,w^2)=k_0,
\end{equation}
for all $w^1,w^2\in \omega(u), w^1\not=w^2.$ 
We now apply Theorem \ref{pb} (Poincar\'{e}-Bendixson). If $\omega(u)$ consists of a single periodic orbit, then (\ref{6.1}) holds by Lemma \ref{lemma3.3mp}. We may therefore assume for the remainder of the proof that for every $w\in \omega(u)$ we have $\alpha(w),\omega(w)\subseteq E.$
If either $w^1$ or $w^2$ is an equilibrium then (\ref{6.1}) holds with $k_0$ defined in Lemma \ref{3.4}. We may therefore assume that $w^1\not\in E.$ 
Suppose now, by contradiction, that there exist $(w^1,w^2) \in \Sigma\setminus \Delta.$ By Axioms (\ref{thingeneral}) and (\ref{Wdropping}) there exists an $\eps_0>0$ such that 
  $\bigl(\phi^{\sigma}(w^1),\phi^{\sigma}(w^2)\bigr)\not\in\Sigma$,  
for all $\sigma \in (-\eps_0,\eps_0)\setminus \{0\}$   
  and
  $$
W(\phi^{\sigma'}(w^1),\phi^{\sigma'}(w^2))<W(\phi^{\sigma}(w^1),\phi^{\sigma}(w^2)),
$$ 
for all $\sigma \in (-\eps_0,0)$ and all $\sigma' \in (0,\eps_0)$.
Set $\sigma = -\eps$ and $\sigma' = \eps$, with $0<\eps <\eps_0$.
Since $w^1\in\omega(u),$ 
 there exists $\sigma_n\to\infty$ such that
$$
w^1=\lim_{n\to\infty}\phi^{\sigma_n}(u),
$$
and
$$
0<\sigma_{n+1}-\sigma_n\to\infty, \ \  \text{as}\ \ n \to\infty.
$$
Define $\hat \sigma_{n}:=(\sigma_{n+1}-\sigma_n)\to\infty$. Then, passing to a subsequence if necessary, the limits
$$
e:=\lim_{n\to\infty}\phi^{-\hat\sigma_n}(\phi^{-\eps}(w^2)) \quad \text{and} \quad \tilde{e}:=\lim_{n\to\infty}\phi^{\hat\sigma_n}(\phi^{\eps}(w^2)) 
$$
exist, and $e,\tilde{e}\in E,$ since $\alpha(w^2)\subseteq E$ and $\omega(w^2)\subseteq E.$ 
By Axiom (\ref{continuityofW}), Lemma \ref{decreasing lemma},  Lemma \ref{3.4} and the fact that $w^1\not\in E$ 
we infer that, for $n$ sufficiently large (slightly shifting $\eps$ if necessary to make $W$ well-defined for all relevant pairs)
$$
\begin{array}{lll}
W(\phi^{\eps}(w^1),\phi^{\eps}(w^2))&<& W(\phi^{-\eps}(w^1),\phi^{-\eps}(w^2))\\
&=&W(\phi^{\sigma_{n+1}-\eps}(u),\phi^{-\eps}(w^2))\\
&\le& W(\phi^{\sigma_{n+1}-\hat{\sigma}_n-\eps}(u),\phi^{-\hat{\sigma}_n-\eps}(w^2))\\
&=&W(\phi^{\sigma_n-\eps}(u),e)\\
&=& W(\phi^{-\eps}(w^1),e)\\
&=& W(\phi^{-\eps}(w^1),\tilde{e})\\
&=& W(\phi^{\eps}(w^1),\tilde{e})\\
&=& W(\phi^{\sigma_{n+1}+\eps}(u),\tilde{e})\\
&=&W(\phi^{\hat{\sigma}_n+\sigma_n+\eps}(u),\phi^{\hat{\sigma}_n+\eps}(w^2))\\
&\le&W(\phi^{\sigma_n+\eps}(u),\phi^{\eps}(w^2))\\
&=&W(\phi^{\eps}(w^1),\phi^{\eps}(w^2)),
\end{array}
$$
which is a contradiction. In the sixth and in the seventh equality we used Lemma \ref{3.4}.
\end{proof}
\end{proposition}

Since the Cauchy-Riemann Equations satisfy the Axioms (A1)-(A5) Theorem \ref{teorema2} follows from Proposition \ref{th2fmp}.

\section{Proofs of Propositions \ref{Xcompact} and \ref{prop:aron}} 
\label{proofoflemmas}
Consider the operators
$$
\partial=\partial_s-J\partial_t\ \  \text{and}\ \ \overline{\partial}=\partial_s+J\partial_t,
$$
and recall the following regularity estimates:
\begin{lemma}\label{lemmaabb}
Let $g$ be a function in $\in C^{\infty}_c(\R\times S^1;\R^2).$ For every $1<p<\infty$, there exists a constant $C_p>0$, such that
\begin{equation}\label{abb}
\Vert \nabla g\Vert_ {L^p(\R\times S^1)}\leq C_p \Vert \bar{\partial}g\Vert_ {L^p(\R\times S^1)}.
\end{equation} 
The same estimate holds for $\partial$ via $t\mapsto -t$.
\end{lemma}	

\begin{proof}
See \cite{abbondandolo}, \cite{niremdouglis}
\cite{hoferzehnder}, \cite[appendix B]{mcduffsalamon}.
\end{proof}

\begin{proof}[Proof of Proposition \ref{Xcompact}]
For a solution $u\in X,$ we can write 
\begin{equation}\label{lincr}
\overline{\partial}u=-JF(t,u)=f(s,t),
\end{equation}
where $F,$ and therefore $f$, are uniformly bounded since for every $u\in X$ we have 
\begin{equation}\label{estimateLinfty}
\Vert u\Vert_ {L^{\infty}(\R\times S^1)}\leq 1.
\end{equation}
 Extend $f$ and $u$  via periodic extension to a function on $\R^2$ in the $t$-direction. By $L^\infty$-bound on $u$ 
 we obtain  the existence of a constant $M>0,$ such that
\begin{equation}\label{festimateLinfty}
\Vert f\Vert_ {L^\infty(\R^2)}\leq M.
\end{equation}

 Let $K,L, G$ be compact sets contained in $\R^2$ such that $K\Subset L \Subset G \subset \mathbb{R}^2,$  and let $\eps$ be positive such that $\eps<\dist(L,\partial G).$ By compactness, $L$ can be covered by finitely many open balls of radius $\eps/2:$ 
 $$
 L\subset \bigcup_{i=1}^{N_\eps} B_{\eps/2}(x_i).
 $$
Consider a partition of unity $\{\rho_{\eps,x_i}\}_{i=1,\dots, N_\eps}$ on $L$ subordinate to  $\left\{B_{\eps}(x_i)\right\}_{i=1,\dots,N_\eps} \!\! .$
In particular the supports of $\rho_{\eps,x_i}$ are contained in $ B_{\eps}(x_i),$ for every $i=1\dots N_{\eps}.$
Then, for every $u,$  every small $\eps>0$ and  every $i=1\dots N_{\eps},$ the function $v_{\eps,i}:=\rho_{\eps,x_i}u $ belongs to $ W^{k,p}_0(\R^2),$ for every $p\geq1,$ and every $k\in\N.$ 
For such functions the 
Poincar\'{e} inequality $\Vert v_{\eps,i}\Vert_{L^p(B_\eps(x_i))} \le C \Vert \nabla v_{\eps,i}\Vert_{L^p(B_\eps(x_i))}$ holds. Combining the latter
with Lemma \ref{lemmaabb} yields (with $C$ changing from line to line) 
 \begin{equation}\label{stima1}
 \begin{array}{ll}
\Vert v_{\eps,i}\Vert_ {W^{1,p}(\R^2)}&= \Vert v_{\eps,i}\Vert_ {W^{1,p}(B_\eps(x_i))} \leq C\Vert v_{\eps,i}\Vert_ {W^{1,p}_0(B_\eps(x_i))}\\
                                                                                  & \leq C\Vert \overline{\partial}v_{\eps,i}\Vert_ {L^p(B_\eps(x_i))}\\
                                                                                  & \leq C\Vert \rho_{\eps, x_i}\overline{\partial}u\Vert_ {L^p(B_\eps(x_i))}+
                                                                                                C\Vert u\overline{\partial}\rho_{\eps, x_i}\Vert_ {L^p(B_\eps(x_i))}\\                                                                       
                                                                                  & \leq C \Vert \overline{\partial}u\Vert_ {L^p(G)}+C\Vert u\Vert_ {L^p(G)}.             
 \end{array}
 \end{equation}
 As $\left\{\rho_{\eps, x_i}\right\}_{i=1,\dots, N_{\eps}}$ is a partition of unity it follows that
 \begin{equation}\label{stima2}
 \Vert u\Vert_ {W^{1,p}(L)}=\left|\left|\sum_{i=1}^{N_{\eps}}v_{\eps,i}\right|\right|_{W^{1,p}(L)}\leq
 \sum_{i=1}^{N_{\eps}}\Vert v_{\eps,i}\Vert_{W^{1,p}(B_\eps(x_i))}.
 \end{equation}
 By (\ref{stima1}) and (\ref{stima2}) we obtain
\begin{equation}\label{3.9 di wojciech}
\Vert u\Vert_ {W^{1,p}(L)}\leq C_{p,L,G}\left(\Vert \overline{\partial}u\Vert_ {L^p(G)}+\Vert u\Vert_ {L^p(G)}\right).
\end{equation}
Combining (\ref{3.9 di wojciech}) with (\ref{lincr}), (\ref{estimateLinfty}) and (\ref{festimateLinfty}) yields
\begin{equation}\label{estimateW1}
\Vert u\Vert_ {W^{1,p}(L)}\leq C_{p,L,G}\left(\Vert f\Vert_ {L^p(G)}+\Vert u\Vert_ {L^p(G)}\right)\le C_{p,L,G}^1,
\end{equation}
where the constant $C_{p,L,G}^1$ depends on $p,L,G$, but not on $u.$ 
By the compactness of the Sobolev compact embedding
$W^{1,p}(L)\hookrightarrow C^{0}(L)$, cf.\ \cite{adams},  sequences $\{u^n\}\subset X$ have convergent subsequences in 
 $C^0_{\text{loc}}(L).$ 
 Since the latter holds  for every $L\subset \R^2,$ 
 the convergence is in  $C^0_{\text{loc}}(\R^2)$ and the limit $u$ is a continuous function.
%
It remains to show that the limit $u$ solves Equation (\ref{nonlinearCR}).  
Consider  a partition of unity of $K\Subset L,$ denoted $\{\rho_{\eps, x_i}\}_{i=1,\dots N_{\eps}}$, where  
 $0<\eps<\dist(K,\partial L).$ On balls $B_{\eps}(x_i)$ we obtain
$$
\begin{array}{ll}
\Vert \rho_{\eps,x_i}u\Vert_ {W^{2,p}(B_{\eps})} &\leq  C\Vert \rho_{\eps, x_i}u\Vert_ {W^{2,p}_0(B_{\eps}(x_i))}\leq C\Vert \overline{\partial}(\rho_{\eps,x_i}u)\Vert_ {W^{1,p}(B_{\eps}(x_i))}\nonumber \\
&\leq  C\displaystyle \left(\Vert \rho_{\eps,x_i}\overline{\partial}u\Vert_ {W^{1,p}(B_{\eps}(x_i))}+\Vert u\overline{\partial}\rho_{\eps, x_i}\Vert_ {W^{1,p}(B_{\eps}(x_i))}\right)\nonumber \\
& \leq  C\left(\Vert \overline{\partial}u\Vert_ {L^\infty(L)}+\Vert \overline{\partial}u\Vert_ {W^{1,p}(L)}+\Vert u\Vert_ {L^\infty(L)}+\Vert u\Vert_ {W^{1,p}(L)}\right).\nonumber
\end{array}
$$
As in (\ref{stima2}), using (\ref{lincr}), we obtain
\begin{equation}\label{stimaW2}
\Vert u\Vert_ {W^{2,p}(K)}\leq \tilde{C}_{p,K,L,G}\left(\Vert f\Vert_ {L^\infty(L)}+\Vert f\Vert_ {W^{1,p}(L)}+\Vert u\Vert_ {L^\infty(L)}+\Vert u\Vert_ {W^{1,p}(L)}\right).
\end{equation}
To estimate the three terms $\Vert f\Vert_ {L^\infty(L)},$ $\Vert u\Vert_ {L^\infty(L)}$ and $\Vert u\Vert_ {W^{1,p}(L)}$
we use (\ref{estimateLinfty}), (\ref{festimateLinfty}),  and (\ref{estimateW1}). 
In order to control $\Vert f\Vert_ {W^{1,p}(L)}$, differentiate the smooth vector field $F$:
\begin{eqnarray}
f_s(s,t)=(F(t,u))_s= D_{t,u} X(t,u)(0,u_s)\nonumber\\
f_t(s,t)=(F(t,u))_t= D_{t,u} X(t,u)(1,u_t).\nonumber
\end{eqnarray}
Both right hand sides lie in $L^{p}(L)$ and hence  $Df=(f_s,f_t)$ is in $L^p(L).$ By  (\ref{stimaW2})
 there exists a constant ${C}_{p,K,L,G}^2$ independent of $u$ such that
$$
\Vert u\Vert_ {W^{2,p}(K)}\leq {C}_{p,K,L,G}^2.
$$ 
By taking $p>2$ the compact Sobolev embedding
$W^{2,p}(K)\hookrightarrow C^{1}({K})$
implies that  $u\in X.$ 
\end{proof}

\begin{proof}[Proof of Proposition \ref{prop:aron}]
As in the proof of Lemma \ref{injectivityofpi1} if suffices to show that $\iota$ is injective.
Suppose there exist $u_1,u_2\in X$ such that $\iota(u_1) = \iota(u_2). $ By definition of $\iota$ we have
\begin{equation}\label{defiota}
u_1(0,\cdot)  = u_2(0,\cdot).
\end{equation}
Define $v(s,t):=u_1(s,t)-u_2(s,t),$ for all $(s,t)\in \R\times S^1.$ 
By (\ref{defiota}) we have $v(0,t)=0$ for all $t\in S^1.$
By smoothness of the vector field $F$ we can write
$$
F(t,u_1)=F(t,u_2)+R(t,u_1,u_2-u_1)(u_2-u_1),
$$
where $R_1$ is a smooth function of its arguments. Upon substitution this gives 
\begin{equation}\label{crlin1}
v_{s}-Jv_{t}+A(s,t)v=0, \quad v(0,t)=0\ \ \  \text{for all}\ t\in S^1,
\end{equation}
and $A(s,t)=R(t,u_1(s,t),v(s,t))$ is (at least) continuous on $\R\times S^1.$ Evaluating (\ref{crlin1}) at $t=0$ we obtain, 
\begin{equation}\label{crlin2}
v_{s}-Jv_{t}+A(s,t)v=0, \quad v(0,0)=0
\end{equation}
Introducing complex coordinates $z:=s+it,$ (\ref{crlin2}) becomes
\begin{equation}\label{crcomplex}
\partial_{}v+A(z)v=0, \quad v(0)=0,
\end{equation}
where the operator $\partial_{}:=\partial_s-i\partial_t$ is the standard anti-holomorphic derivative. We used the identification between the complex structure $J$ in $\R^2$ and $i$ in $\mathbb{C}.$ 
Multiplying (\ref{crcomplex}) by $e^{\int_0^{z}A(\zeta)d\zeta}$ and defining
$$
w(z):=e^{\int_{0}^{z}A(\zeta)d\zeta}v(z),
$$
gives
$$
\partial_{} w=0, \quad w(0)=0.
$$
which implies that $w$ is analytic. The latter yields that either $0$ is an isolated zero for $w$, or there exists a $\delta>0,$ such that $w(z)=0,$ on $U_{\delta}:=\{z\in \mathbb{C}:|z|\leq \delta\}.$ By (\ref{crlin1}) we conclude that 0 cannot be an isolated zero for $w$, hence $w\equiv 0$ in $U_{\delta}:=\{z\in \mathbb{C}:|z|\leq \delta\}.$ Repeating these arguments we obtain that $w(s,t)=0 $ for all $(s,t)\in \R\times S^1$ and hence $v\equiv0.$ This implies $u_1=u_2,$ which concludes the proof.\end{proof}

\begin{remark}
{\em
The same proof can be carried out in case $J$ is a smooth map $\R\times S^1\to \Sp(2,\R)$ such that $J^2=-\Id$.
In this case one can prove that the equation 
$
u_s-J(s,t)(u_t-F(t,u))=0
$
can be transformed into (\ref{nonlinearCR}) using \cite[Theorem 12, Appendix A.6]{hoferzehnder}.
}
\end{remark}

\bibliographystyle{amsplain}

\bibliography{bibth}

\end{sloppypar}

\end{document}